\documentclass[12pt]{amsart}
\usepackage{amsmath,amsfonts,amssymb,amsthm}
\usepackage{graphics}
\usepackage{epsfig,psfrag}

\newtheoremstyle{thm}
  {9pt}{9pt}{\itshape}{}{\bfseries}{}{.5em}{}
\theoremstyle{thm}
\newtheorem{thm}{Theorem}

\newtheorem{lemma}[thm]{Lemma}
\newtheorem{prop}[thm]{Proposition}
\newtheorem*{conj}{Conjecture}

\newtheoremstyle{defin}
  {9pt}{9pt}{}{}{\bfseries}{}{.5em}{}
\theoremstyle{defin}

\newtheoremstyle{exm}
  {9pt}{9pt}{}{}{\scshape}{}{.5em}{}
\theoremstyle{exm}
\newtheorem*{exm}{Example}

\newtheoremstyle{proof}
  {}{}{}{}{\itshape}{:}{.5em}{}
\theoremstyle{proof}

\newtheorem*{skt}{Sketch of proof}

\newcommand{\set}[1]{\{#1\}}
\newcommand{\Z}{{\mathbb Z}}
\newcommand{\p}[1]{\mathcal{#1}}
\newcommand{\s}[1]{\mathbf{#1}}

\newcommand{\pH}{\overline{\mathcal H}}

\author{Matja\v z Konvalinka$^*$}
\title[Weighted hook length formula III]{The weighted hook length formula III:\\Shifted tableaux}

\thanks{${\hspace{-.95ex}}^*$Department of Mathematics,
Vanderbilt University, Nashville;~\texttt{matjaz.konvalinka@vanderbilt.edu}}

\begin{document}

\begin{abstract}
 Recently, a simple proof of the hook length formula was given via the branching rule. In this paper, we extend the results to shifted tableaux. We give a bijective proof of the branching rule for the hook lengths for shifted tableaux; present variants of this rule, including weighted versions; and make the first tentative steps toward a bijective proof of the hook length formula for $d$-complete posets.
\end{abstract}

\maketitle

\section{Introduction}

Let $\lambda = (\lambda_1,\lambda_2,\ldots,\lambda_\ell)$, $\lambda_1 \geq \lambda_2 \geq \ldots \geq \lambda_\ell > 0$, be a partition of $n$, $\lambda \vdash n$, and let $[\lambda] = \set{(i,j) \in \Z^2 \colon 1 \leq i \leq \ell, 1 \leq j \leq \lambda_i}$ be the corresponding \emph{Young diagram}. The \emph{conjugate partition} $\lambda'$ is defined by $\lambda_j' = \max \{i : \lambda_i \geq j\}$. The \emph{hook} $H_{\s z} = H_{ij} \subseteq [\lambda]$ is the set of squares weakly to the right and below of $\s z = (i,j) \in [\lambda]$, and the {\it hook length} $h_{\s z} = h_{ij} = |H_{\s z}|= \lambda_i +\lambda'_j - i - j +1$ is the size of the hook. The \emph{punctured hook} $\overline H_{\s z} \subseteq [\lambda]$ is the set $H_{\s z} \setminus \set{\s z}$. See Figure \ref{fig1}, left and center drawing.

\medskip

We write $[n]$ for $\set{1,\ldots,n}$, $\p P(A)$ for the power set of $A$, and $A \sqcup B$ for the disjoint union of (not necessarily disjoint) sets $A$ and $B$. Furthermore, for squares $(i,j)$ and $(i',j')$ of $[\lambda]$, write $(i,j) \geq (i',j')$ if $i \leq i'$ and $j \leq j'$.

\medskip

A \emph{standard Young tableau} of shape $\lambda$ is a bijective map $f: [\lambda] \to [n]$, such that $f(\s z) < f(\s {z'})$ whenever $\s z \geq \s {z'}$ and $\s z \neq \s {z'}$. See Figure \ref{fig1}, right drawing. We denote the number of standard Young tableaux of shape $\lambda$  by $f_\lambda$. The remarkable hook length formula states that if $\lambda$ is a partition of $n$, then
$$f_\lambda = \frac{n!}{\prod_{\s z \in [\lambda]} h_{\s z}}.$$

For example, for $\lambda=(3,2,2)\vdash 7$, the hook length formula gives
$$f_{322} \, = \, \frac{7!}{5\cdot 4 \cdot 1 \cdot 3\cdot 2 \cdot 2 \cdot 1} \, = \, 21.$$

\begin{figure}[hbt]
\psfrag{1}{$1$}
\psfrag{2}{$2$}
\psfrag{3}{$3$}
\psfrag{4}{$4$}
\psfrag{5}{$5$}
\psfrag{6}{$6$}
\psfrag{7}{$7$}
\psfrag{L}{$\lambda$}
\begin{center}
\epsfig{file=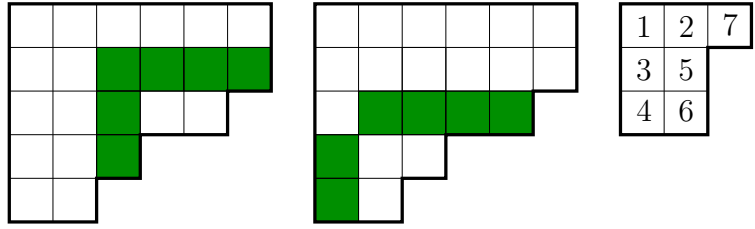,height=3cm}
\end{center}
\caption{Young diagram $[\lambda]$,  $\lambda = 66532$, hook $H_{23}$ with hook length $h_{23}=6$ (left), punctured hook $\overline H_{31}$ (center); a standard Young tableau of shape $322$ (right).}
\label{fig1}
\end{figure}

This gives a short formula for dimensions of irreducible representations of the symmetric group, and is a fundamental result in algebraic combinatorics.  The formula was discovered by Frame, Robinson and Thrall in~\cite{FRT} based on earlier resultsof Young, Frobenius and Thrall.  Since then, it has been reproved, generalized and extended in several different ways, and applications have been found in a number of fields ranging from algebraic geometry to probability, and from group theory to the analysis of algorithms.

\medskip

One way to prove the hook length formula is by induction on $n$. Namely, it is obvious that in a standard Young tableau, $n$ must be in one of the \emph{corners}, squares $(i,j)$ of $[\lambda]$ satisfying $(i+1,j),(i,j+1) \notin [\lambda]$. Therefore
$$f_\lambda  =  \sum_{\s c \in \p C[\lambda]}  f_{\lambda-\s c},$$
where $\p C[\lambda]$ is the set of all corners of $\lambda$, and $\lambda-\s c$ is the partition whose diagram is $[\lambda] \setminus \set{\s c}$.

\medskip

That means that in order to prove the hook length formula, we have to prove that $F_\lambda = n!/\prod h_{\s z}$ satisfy the same recursion. It is easy to see that this is equivalent to the following \emph{branching rule for the hook lengths}:
\begin{equation} \label{branch}
  n \cdot \!\!\!\prod_{(i,j) \in [\lambda] \setminus \p C[\lambda]} (h_{ij} - 1)\, = \, \sum_{(r,s) \in \p C[\lambda]}\left[\prod_{\stackrel{(i,j) \in [\lambda]\setminus \p C[\lambda]}{\scriptscriptstyle i \neq r,j \neq s}} (h_{ij} - 1)\right] \prod_{i=1}^{r-1} h_{is} \prod_{j=1}^{s-1} h_{rj}.
\end{equation}

In \cite{ckp}, a weighted generalization of this formula was presented, with a simple bijective proof.

\medskip

Suppose that $\lambda = (\lambda_1,\lambda_2,\ldots,\lambda_\ell)$, $\lambda_1 > \lambda_2 > \ldots > \lambda_\ell > 0$, is a partition of $n$ \emph{with distinct parts}, and let $[\lambda]^* = \set{(i,j) \in \Z^2 \colon 1 \leq i \leq \ell, i \leq j \leq \lambda_i+i-1}$ be the corresponding \emph{shifted Young diagram}. The \emph{hook} $H_{\s z}^* = H_{ij}^* \subseteq [\lambda]^*$ is the set of squares weakly to the right and below of $\s z = (i,j) \in [\lambda]$, and in row $j+1$, and the {\it hook length} $h_{\s z}^* = h_{ij}^* = |H_{\s z}^*|$ is the size of the hook. It is left as an exercise for the reader to check that $h_{ij}^* = \lambda_i + \lambda_{j+1}$ if $j < \ell(\lambda)$ and $h_{ij}^* = \lambda_i + \max\set{k \colon \lambda_k \geq j + 1 - k \geq 1} - j$ if $j \geq \ell(\lambda)$. The \emph{punctured hook} $\overline H_{\s z}^* \subseteq [\lambda]^*$ is the set $H^*_{\s z} \setminus \set{\s z}$. See Figure \ref{fig2}, left and center drawing. For squares $(i,j)$ and $(i',j')$ of $[\lambda]^*$, write $(i,j) \geq (i',j')$ if $i \leq i'$ and $j \leq j'$.

\medskip

A \emph{standard shifted Young tableau} of shape $\lambda$ is a bijective map $f: [\lambda]^* \to [n]$, such that $f(\s z) < f(\s {z'})$ whenever $\s z \geq \s {z'}$ and $\s z \neq \s {z'}$. See Figure \ref{fig2}, right drawing. We denote the number of standard shifted Young tableaux of shape $\lambda$  by $f^*_\lambda$. The shifted hook length formula states that if $\lambda$ is a partition of $n$ with distinct parts, then
$$f^*_\lambda = \frac{n!}{\prod_{\s z \in [\lambda]} h_{\s z}^*}.$$

For example, for $\lambda=(5,3,2)\vdash 10$, the hook length formula gives
$$f^*_{532} \, = \, \frac{10!}{8\cdot 7 \cdot 5 \cdot 4\cdot 1 \cdot 5 \cdot 3\cdot 2\cdot 2\cdot 1} \, = \, 54.$$

\begin{figure}[hbt]
\psfrag{1}{$1$}
\psfrag{2}{$2$}
\psfrag{3}{$3$}
\psfrag{4}{$4$}
\psfrag{5}{$5$}
\psfrag{6}{$6$}
\psfrag{7}{$7$}
\psfrag{8}{$8$}
\psfrag{9}{$9$}
\psfrag{10}{$10$}
\psfrag{L}{$\lambda$}
\begin{center}
\epsfig{file=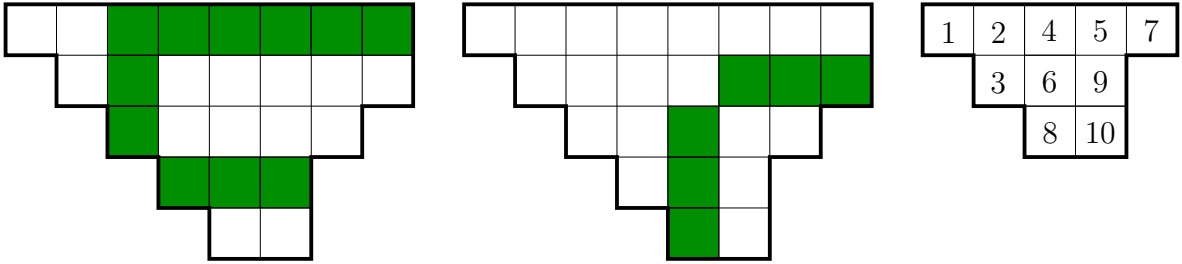,height=3.5cm}
\end{center}
\caption{Shifted Young diagram $[\lambda]$, $\lambda = 87532$, a hook $H_{13}^*$ with hook length $h_{13}^*=11$ (left), a punctured hook $\overline H_{25}^*$ (center); a standard shifted Young tableau of shape $532$ (right).}
\label{fig2}
\end{figure}

Again, one way to prove the shifted hook length formula is by induction on $n$. Namely, it is obvious that in a standard shifted Young tableau, $n$ must be in one of the \emph{shifted corners}, squares $(i,j)$ of $[\lambda]^*$ satisfying $(i+1,j),(i,j+1) \notin [\lambda]^*$. Therefore
$$f^*_\lambda  =  \sum_{\s c \in \p C^*[\lambda]}  f^*_{\lambda-\s c},$$
where $\p C^*[\lambda]$ is the set of all shifted corners of $\lambda$, and $\lambda-\s c$ is the partition whose shifted diagram is $[\lambda]^* \setminus \set{\s c}$.

\medskip

That means that in order to prove the shifted hook length formula, we have to prove that $F^*_\lambda = n!/\prod h_{\s z}^*$ satisfy the same recursion. It is easy to see that this is equivalent to the following \emph{shifted branching rule for the hook lengths}:
\begin{equation} \label{sh-branch}
  n \cdot \!\!\!\!\!\prod_{(i,j) \in [\lambda]^* \setminus \p C^*[\lambda]} (h^*_{ij} - 1)\, = \, \sum_{(r,s) \in \p C^*[\lambda]}\left[\prod_{\stackrel{(i,j) \in [\lambda]\setminus \p C[\lambda]}{\scriptscriptstyle i \neq r,j \neq r-1,s}} (h^*_{ij} - 1)\right] \prod_{i=1}^{r-1} h^*_{is} \prod_{j=1}^{s-1} h^*_{rj} \prod_{i=1}^{r-1} h^*_{i,r-1}.
\end{equation}

Interestingly, the shifted hook length formula was discovered before the non-shifted one \cite{thr}. Shortly after the famous Greene-Nijenhuis-Wilf's probabilistic proof of the ordinary hook length formula \cite{gnw}, Sagan \cite{sagan} extended the argument to the shifted case. The proof, however, is rather technical. In 1995, Krattenthaler \cite{Kra1} provided a bijective proof. While short, it is very involved, as it needs a variant of Hillman-Grassl algorithm, a bijection that comes from Stanley's $(P,\omega)$-partition theory, and the involution principle of Garsia and Milne. A few years later, Fischer \cite{fischer} gave the most direct proof of the formula, in the spirit of Novelli-Pak-Stoyanovskii's bijective proof of the ordinary hook-length formula. At almost 50 pages in length, the proof is very involved. Bandlow \cite{Ban} gave a short proof via interpolation.

\medskip

The main result of this paper is a bijective proof of the branching formula for the hook lengths for shifted tableaux, which is trivially equivalent to the hook length formula. While the proof still cannot be described as simple, it seems that it is the most intuitive of the known proofs. Once the proof in the non-shifted case is understood in a proper context, the bijection for shifted tableaux has only two (main) extra features: a description of a certain map $s$, described by (S\ref{s1})--(S\ref{s5}) on page \pageref{s1}, and the process of ``flipping the snake'' of Lemma \ref{flip} with a half-page proof.

\medskip

This paper is organized as follows. Section \ref{non-shifted} presents the non-weighted bijection from \cite{ckp} in a new way, to make it easier to adapt to the shifted case. Section \ref{basic} discusses the basic features of the bijective proof of \eqref{sh-branch}, and Section \ref{details} gives all the requisite details. Sections \ref{basic} and \ref{details} thus give the first completely bijective proof of the shifted branching rule. In Section \ref{comparison}, we discuss the above-mentioned paper by Sagan \cite{sagan}. Section \ref{variants} gives variants of the formulas in the spirit of \cite[\S 2]{ckp}, including weighted formulas. In Section \ref{d-complete}, we make the first step to extending the bijective proof to $d$-complete posets. We conclude with final remarks in Section \ref{final}. 

\section{Bijection in the non-shifted case} \label{non-shifted}

In this section, we essentially describe the bijection that was used in \cite{ckp} to prove the weighted branching rule for hook lengths (see also \cite{zei}). The reader is advised to read Subsection 2.2 of that paper first, however, as the description that follows is much more abstract. This way, it is more easily adaptable to the shifted case.

\medskip

Let us give an interpretation for each side of equality \eqref{branch}. The left-hand side counts all pairs $(\s s, F)$, where $\s s \in [\lambda]$ and $F$ is a map from $[\lambda] \setminus \p C[\lambda]$ to $[\lambda]$ so that $F(\s z)$ is in the punctured hook of $\s z$ for every $\s z \in [\lambda] \setminus \p C[\lambda]$. We also write $F_{\s z} = F_{i,j}$ for $F(\s z)$ when $\s z = (i,j)$. We think of the map $F$ as an arrangement of labels; we label the square $(i,j)$ by $F_{i,j} = (k,l)$ or sometimes just $kl$. The left drawing in Figure \ref{fig3} shows an example of such a pair $(\s s, F)$; the square $\s s$ is drawn in green. Denote the set of all such pairs by $\p F$. 

\begin{figure}[hbt]
\psfrag{f1}{$31$}

\psfrag{f2}{$32$}

\psfrag{f3}{$16$}

\psfrag{f4}{$34$}

\psfrag{f5}{$16$}

\psfrag{f6}{$26$}

\psfrag{f7}{$25$}

\psfrag{f8}{$23$}

\psfrag{f9}{$33$}

\psfrag{f10}{$34$}

\psfrag{f11}{$26$}

\psfrag{f12}{$\phantom{11}$}

\psfrag{f13}{$34$}

\psfrag{f14}{$34$}

\psfrag{f15}{$34$}

\psfrag{f16}{$35$}

\psfrag{f17}{$\phantom{11}$}

\psfrag{f18}{$51$}

\psfrag{f19}{$52$}

\psfrag{f20}{$\phantom{11}$}

\psfrag{f21}{$52$}

\psfrag{f22}{$\phantom{11}$}

\psfrag{f1g}{$21$}

\psfrag{f2g}{$13$}

\psfrag{f3g}{$33$}

\psfrag{f4g}{$34$}

\psfrag{f5g}{$15$}

\psfrag{f6g}{$26$}

\psfrag{f7g}{$24$}

\psfrag{f8g}{$25$}

\psfrag{f9g}{$33$}

\psfrag{f10g}{$26$}

\psfrag{f11g}{$35$}

\psfrag{f12g}{$\phantom{11}$}

\psfrag{f13g}{$31$}

\psfrag{f14g}{$32$}

\psfrag{f15g}{$35$}

\psfrag{f16g}{$35$}

\psfrag{f17g}{$\phantom{11}$}

\psfrag{f18g}{$43$}

\psfrag{f19g}{$52$}

\psfrag{f20g}{$\phantom{11}$}

\psfrag{f21g}{$52$}

\psfrag{f22g}{$\phantom{11}$}
\psfrag{L}{$\lambda$}
\begin{center}
\epsfig{file=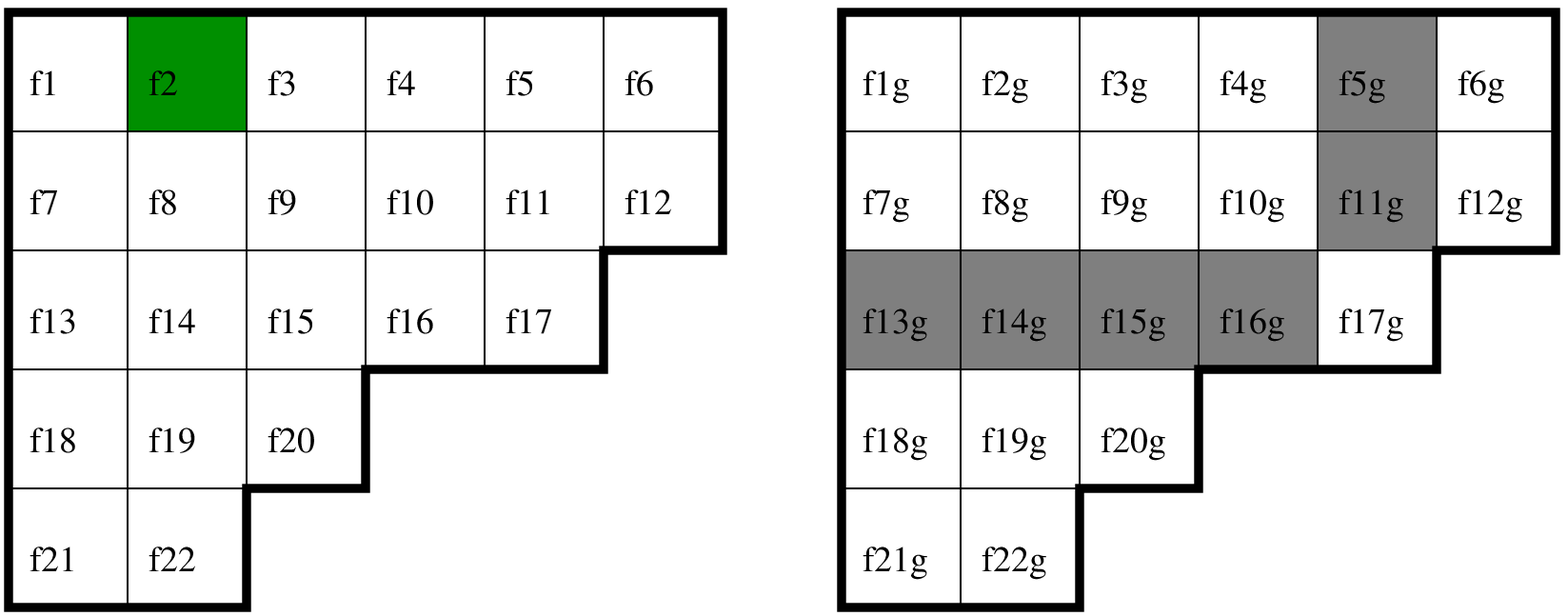,height=3.8cm}
\end{center}
\caption{A representative of the left-hand side and the right-hand side of \eqref{branch} for $\lambda = 66532$.}
\label{fig3}
\end{figure}

\medskip

Similarly, the right-hand side of \eqref{branch} counts the number of pairs $(\s c, G)$, where $\s c \in \p C[\lambda]$ and $G$ is a map from $[\lambda] \setminus \p C[\lambda]$ to $[\lambda]$ satisfying the following:
\begin{itemize}
 \item if $\s c$ lies in the hook of $\s z$, then $G(\s z)$ lies in the hook of $\s z$;
 \item otherwise, $G(\s z)$ lies in the punctured hook of $\s z$.
\end{itemize}
Note that if $\s c = (r,s)$, then $\s c$ lies in the hook of $\s z = (i,j)$ if and only if $i = r$ or $j = s$. We think of the map $G$ as an arrangement of labels. The right drawing in Figure \ref{fig3} shows an example of such a pair $(\s c, G)$; the row $r$ and the column $s$ without the corner $(r,s)$ are shaded. Denote the set of all $G$ such that $(\s c,G)$ satisfies the above properties by $\p G_{\s c}$, and the disjoint union $\bigsqcup_{\s c} \p G_{\s c}$ by $\p G$. We also write $G_{\s z} = G_{i,j}$ for $G(\s z)$ when $\s z = (i,j)$. 

\medskip

Some squares $\s z$ in the shaded row and column may satisfy $G_{\s z} = \s z$. Let us define sets $A,B$ as follows:
\begin{eqnarray*}
 \aligned
 A = A(G) & = \set{i \colon G_{i,s} = (i,s)} \\
 B = B(G) & = \set{j \colon G_{r,j} = (r,j)}
 \endaligned
\end{eqnarray*}

We think of the set $A$ as representing column $s$ without the corner $(r,s)$, and of $B$ as representing row $r$ without the corner $(r,s)$, with a dot in all squares satisfying $G_{\s z} = \s z$. Figure \ref{fig4} shows the shaded row and column for the example from the right-hand side of Figure \ref{fig3}.

\begin{figure}[hbt]

\begin{center}
\epsfig{file=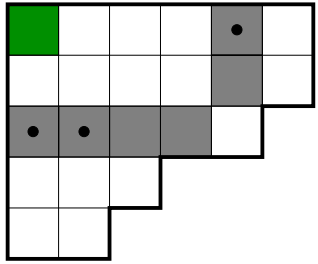,height=3.5cm}
\end{center}
\caption{Shaded row and column and sets $A$ and $B$ for the example on the right-hand side of Figure \ref{fig3}, and $s(A,B)$.}
\label{fig4}
\end{figure}

\medskip

Furthermore, denote by $\p G_{\s c}^k$ the set of all $G \in \p G_{\s c}$ for which exactly $k$ of the squares in $[\lambda] \setminus \p C[\lambda]$ satisfy $G_{\s z} = \s z$ (note that such $\s z$ must be in the shaded row or column); in other words, $|A| + |B| = k$.

\medskip

We construct two maps, $s = s_{\s c}$ and $e = e_{\s c}$. The map $s \colon \p P([r-1]) \times \p P([s-1]) \to [\lambda]$ (``starting square'') is defined by
$$s(A,B) = (\min(A \cup \set r),\min(B \cup \set s)).$$

We can extend it to a function $s \colon \p G_{\s c} \to [\lambda]$ by $s(G) = s(A(G),B(G))$. See Figure \ref{fig4}; there, $s(A,B)$ is green.

\medskip

Now let us define the map $e(G) \colon \p G_{\s c} \to \p G_{\s c}$ (``erasing a dot''). If $G \in \p G_{\s c}^0$, then $e(G) = G$. Otherwise, the arrangement $e(G)$ differs from $G$ in only one or two squares. The first case is when $s(G)$ is a shaded square; this happens if $A(G) = \emptyset$ or $B(G) = \emptyset$. If $A = A(G) = \emptyset$ and $B = B(G) = \set{j,\ldots}$, then $s(G) = (r,j)$; define $A' = A = \emptyset$, $B' = B \setminus \set j$ and $e(G)_{r,j} = s(A',B')$. If $A = A(G) = \set{i,\ldots}$ and $B = B(G) = \emptyset$, then $s(G) = (i,s)$; define $A'= A \setminus \set i$, $B' = B = \emptyset$ and $e(G)_{i,s} = s(A',B')$. In each case, we have $A(e(G)) = A'$, $B(e(G)) = B'$. See Figure \ref{fig15}.

\medskip

The second case is when $s(G) = (i,j)$ for $i < r$ and $j < s$; this happens when $A = A(G)$ and $B = B(G)$ are both non-empty. The crucial observation is that the punctured hook of $(i,j)$, which consists of squares $(i,k)$ for $j < k \leq \lambda_i$ and $(k,j)$ for $i < k \leq \lambda'_j$, is in a natural correspondence with $\overline H_{is} \sqcup \overline H_{rj}$. Indeed, squares of the form $(i,k)$ for $s < k \leq \lambda_i$ are also in the punctured hook of $(i,s)$, and squares of the form $(k,j)$ for $i < k \leq r$ are in a natural correspondence with squares of the form $(k,s)$ for $i < k \leq r$, which are in the punctured hook of $(i,s)$. On the other hand, squares of the form $(i,k)$ for $j < k \leq s$ are in a natural correspondence with squares of the form $(r,k)$ for $j < k \leq s$, which are in the punctured hook of $(r,j)$, and squares of the form $(k,j)$ for $r < k \leq \lambda'_j$ are also in the punctured hook of $(r,j)$. See Figure \ref{fig5}.

\begin{figure}[hbt]
\begin{center}
\epsfig{file=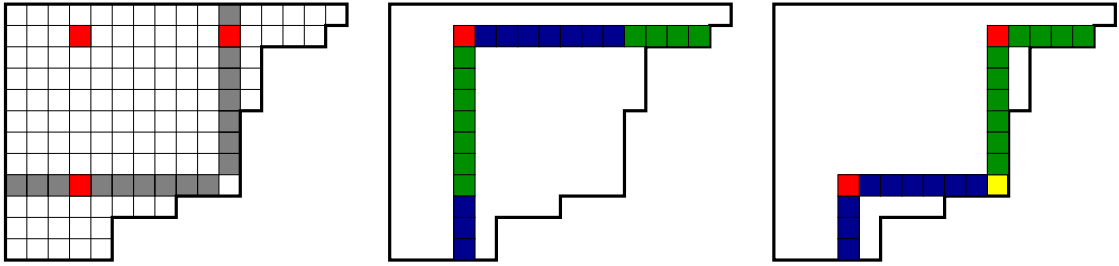,height=3.5cm}
\end{center}
\caption{The punctured hook of a square is the disjoint union of the punctured hooks of projections onto shaded row and column. The yellow square is in the punctured hook of both projections.}
\label{fig5}
\end{figure}

\medskip

That means that by slight abuse of notation, we can say that $G_{i,j} \in \overline H_{is} \sqcup \overline H_{rj}$. If $G_{i,j} \in \overline H_{is}$, define $e(G)_{i,s} = G_{i,j}$ and $e(G)_{i,j} = s(A',B')$, where $A' = A \setminus \set i$ and $B' = B$. If $G_{i,j} \in \overline H_{rj}$, define $e(G)_{r,j} = G_{i,j}$ and $e(G)_{i,j} = s(A',B')$, where $A' = A$ and $B' = B \setminus \set j$. In each case, we have $A(e(G)) = A'$, $B(e(G)) = B'$. See Figure \ref{fig6}.

\begin{figure}[hbt]

\psfrag{f1g}{$31$}

\psfrag{f2g}{$13$}

\psfrag{f3g}{$33$}

\psfrag{f4g}{$34$}

\psfrag{f5g}{$25$}

\psfrag{f6g}{$26$}

\psfrag{f7g}{$24$}

\psfrag{f8g}{$25$}

\psfrag{f9g}{$33$}

\psfrag{f10g}{$26$}

\psfrag{f11g}{$35$}

\psfrag{f12g}{$\phantom{11}$}

\psfrag{f13g}{$31$}

\psfrag{f14g}{$32$}

\psfrag{f15g}{$35$}

\psfrag{f16g}{$35$}

\psfrag{f17g}{$\phantom{11}$}

\psfrag{f18g}{$43$}

\psfrag{f19g}{$52$}

\psfrag{f20g}{$\phantom{11}$}

\psfrag{f21g}{$52$}

\psfrag{f22g}{$\phantom{11}$}
\psfrag{L}{$\lambda$}
\begin{center}
\epsfig{file=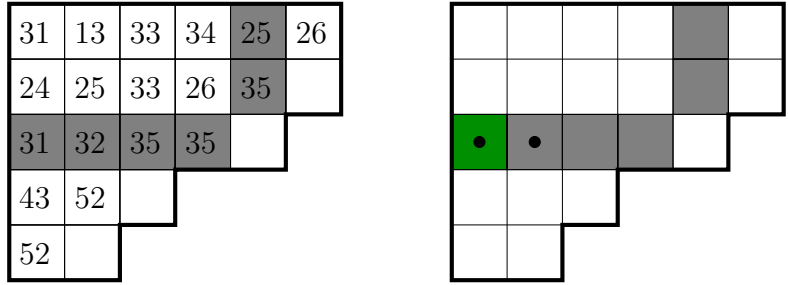,height=3.8cm}
\end{center}
\caption{Arrangement $G' = e(G)$ and corresponding $A'$, $B'$, $s(A',B')$ for $G$ from Figure \ref{fig3}.}
\label{fig6}
\end{figure}

\medskip

\begin{figure}[hbt]

\psfrag{f1g}{$31$}

\psfrag{f2g}{$13$}

\psfrag{f3g}{$33$}

\psfrag{f4g}{$34$}

\psfrag{f5g}{$25$}

\psfrag{f6g}{$26$}

\psfrag{f7g}{$24$}

\psfrag{f8g}{$25$}

\psfrag{f9g}{$33$}

\psfrag{f10g}{$26$}

\psfrag{f11g}{$35$}

\psfrag{f12g}{$\phantom{11}$}

\psfrag{f13g}{$32$}

\psfrag{f14g}{$32$}

\psfrag{f15g}{$35$}

\psfrag{f16g}{$35$}

\psfrag{f17g}{$\phantom{11}$}

\psfrag{f18g}{$43$}

\psfrag{f19g}{$52$}

\psfrag{f20g}{$\phantom{11}$}

\psfrag{f21g}{$52$}

\psfrag{f22g}{$\phantom{11}$}
\begin{center}
\epsfig{file=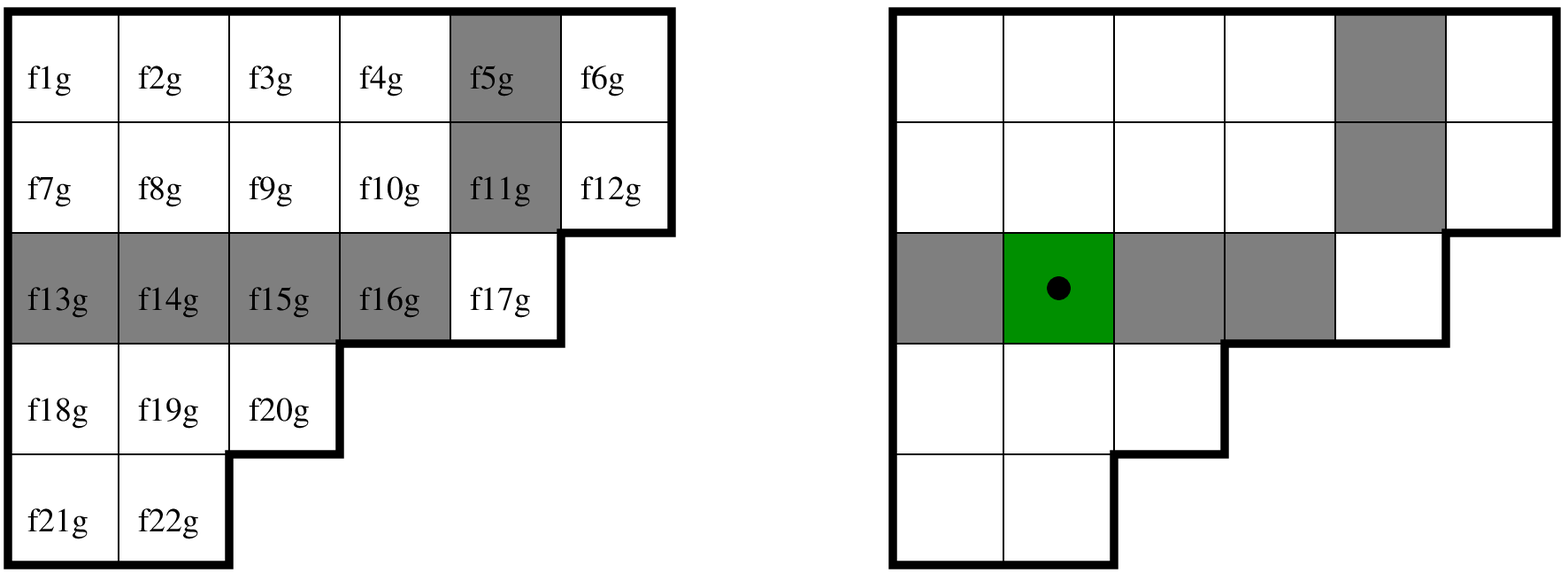,height=3.8cm}
\end{center}
\caption{Arrangement $G'' = e(G')$ and corresponding $A'$, $B'$, $s(A',B')$ for $G'$ from Figure \ref{fig6}.}
\label{fig15}
\end{figure}

\medskip

We claim that the maps $s$ and $e$ have the following three properties:
\begin{enumerate}
 \renewcommand{\labelenumi}{(P\arabic{enumi})}
 \item \label{p1} if $G \in \p G_{\s c}^0$, then $s(G) = \s c$ and $e(G) = G$, and if $G \in \p G_{\s c}^k$ for $k \geq 1$, then $e(G) \in \p G_{\s c}^{k-1}$;
 \item \label{p2} if $G \in \p G_{\s c}^k$ for $k \geq 1$, then $e(G)_{s(G)} = s(e(G))$; furthermore, if $\s z \not\leq s(G)$, then $e(G)_{\s z} = G_{\s z}$;
 \item \label{p3} given $G' \in \p G_{\s c}^k$ and $\s z$ for which $G'_{\s z} = s(G')$, there is exactly one $G \in \p G_{\s c}^{k+1}$ satisfying $s(G) = \s z$ and $e(G) = G'$.
\end{enumerate}

Indeed, the first statement of (P\ref{p1}) is clear from the definition of $s$ and the second statement is part of the definition of $e$. The third statement of (P\ref{p1}) is immediately obvious by inspection of all cases: we always leave one of the sets $A$, $B$ intact, and remove one element from the other. Part (P\ref{p2}) also follows by construction. For example, for $A = \set{i,i',\ldots}$, $B = \set{j,\ldots}$, $G_{ij} \in \overline H_{is}$, $A(e(G)) = \set{i',\ldots} = A'$, $B(e(G)) = \set{j,\ldots} = B'$ and $s(e(G)) = (i',j)$. By definition, $e(G)_{s(G)} = e(G)_{ij} = (i',j) = s(e(G))$. Also, $e(G)$ differs from $G$ only in $(i,j)$ and $(i,s)$, and so $e(G)_{\s z} = G_{\s z}$ if $\s z \not\leq s(G)$. Other cases are very similar. That leaves only (P\ref{p3}). Assume that $A(G') = \set{i',\ldots}$, $B(G') = \set{j',\ldots}$, $s(G') = (i',j')$, $\s z = (i,j')$ for $i < i'$, $G'_{i,j'} = s(G') = (i',j')$. It is easy to see that $G$ defined by $G_{i,j'} = G'_{i,s}$, $G_{i,s} = (i,s)$, $G_{\s z} = G'_{\s z}$ for $\s z \neq (i,j'),(i,s)$ satisfies $s(G) = \s z$ and $e(G) = G'$, and that this is the only $G$ with these properties. Other cases are similar.

\medskip

Observe that by (P\ref{p1}), the sequence $G,e(G),e^2(G),\ldots$ eventually becomes constant, say $\phi(G)$, for every $G$. More specifically, if $G \in \p G_{\s c}^k$, then $\phi(G) = e^k(G)$. Let us first prove a simple consequences of (P\ref{p2}). Take $G \in \p G_{\s c}^k$ for $k \geq 1$. Since $s(e(G))$ is in the punctured hook of $s(G)$, we have $s(G) \not\leq s(e(G))$ and by the last statement in (P\ref{p2}), $e^2(G)_{s(G)} = e(G)_{s(G)}=s(e(G))$. Similarly, $s(G) \not \leq s(e^2(G))$, and so $e^3(G)_{s(G)} = e^2(G)_{s(G)} = s(e(G))$. By induction, we have
\begin{equation} \label{eq1}
 \phi(G)_{s(G)} = s(e(G)).
\end{equation}

\medskip

We claim that $\Phi \colon G \mapsto (s(G),\phi(G))$ is a bijection between $\p G$ and $\p F$. To see that, note first that an element $(\s z,F)$ of $\p F$ naturally gives a ``hook walk'' $\s z \to F(\s z) \to F^2(\s z) \to F^3(\s z) \to \ldots \to \s c$ for some $\s c \in \p C[\lambda]$. For $\s c \in \p C[\lambda]$ and $k \geq 0$, denote by $\p F_{\s c}^k$ the set of all $(\s z,F)$ for which $F^k(\s z) = \s c$. We prove by induction on $k$ that given $(\s z,F) \in \p F_{\s c}^k$, $\Phi(G) = (\s z,F)$ for exactly one $G \in \p G$, and that this $G$ is an element of $\p G_{\s c}^k$.

\medskip

If $(\s z,F) \in \p F_{\s c}^0$ (i.e., if $\s z = \s c$), then $\Phi(G) = (\s c,F)$ if and only if $G = F \in \p G_{\s c}^0$. Now assume we have proved the statement for $0,\ldots,k$, and take $(\s z,F) \in \p F_{\s c}^{k+1}$. Because $(F(\s z),F) \in \p F_{\s c}^k$, we have $(F(\s z),F) = \Phi(G') = (s(G'),\phi(G'))=(s(G'),e^k(G'))$ for (exactly one) $G' \in \p G$, and we have $G' \in \p G_{\s c}^k$. Since $G'_{\s z} = e(G')_{\s z} = e^2(G')_{\s z} = \ldots = F_{\s z} = s(G')$ by an argument similar to the proof of \eqref{eq1}, there is (exactly one) $G \in \p G_{\s c}^{k+1}$ satisfying $s(G) = \s z$ and $e(G) = G'$, by (P\ref{p3}). Obviously, $\Phi(G) = (s(G),\phi(G)) = (\s z,\phi(e(G'))=(\s z,\phi(G')) = (\s z,F)$. If $\Phi(G'') = (\s z,F)$ for some $G'' \in \p G$, then $s(G'') = \s z$ and $\phi(G'') = F$. We have $\Phi(e(G'')) = (s(e(G'')),\phi(e(G''))) = (s(e(G'')),\phi(G'')) = (\phi(G'')_{s(G'')},\phi(G''))$ by equation \eqref{eq1}, and hence $\Phi(e(G'')) = (F(\s z),F) = \Phi(G')$. By induction, $e(G'') = G'$. Since both $s(G'')=s(G) = \s z$ and $e(G'') = e(G) = G'$, we have $G''=G$ by (P\ref{p3}).

\medskip

The message of this section is the following. The crux of the bijective proof of the branching formula is to find the maps $s$ (start) and $e$ (erase). Properties (P\ref{p1}), (P\ref{p2}) and (P\ref{p3}) are easy corollaries of the construction, and we can then prove the bijectivity of the map $\p G \to \p F$, $G \mapsto (s(G),e(e(\cdots(e(G))\cdots)))$, via a completely abstract proof. The branching formula for the hook lengths, equation \eqref{branch}, follows because the right-hand side enumerates $\p G$, and the left-hand side enumerates $\p F$.

\section{Basic features of the bijection} \label{basic}

The descriptions of the left-hand and right-hand sides of equation \eqref{sh-branch} are very similar. The left-hand side counts the number of pairs $(\s s, F)$, where $\s s \in [\lambda]^*$ and $F$ is a map from $[\lambda]^* \setminus \p C^*[\lambda]$ to $[\lambda]^*$ so that $F(\s z)$ is in the punctured hook of $\s z$ for every $\s z \in [\lambda]^* \setminus \p C^*[\lambda]$. We think of the map $F$ as an arrangement of labels. The left drawing in Figure \ref{fig7} shows an example of such a pair $(\s s, F)$; the square $\s s$ is green. Denote the set of all such pairs by $\p F$. We also write $F_{\s z} = F_{i,j}$ for $F(\s z)$ when $\s z = (i,j)$.

\begin{figure}[hbt]
\psfrag{f1}{$14$}

\psfrag{f2}{$33$}

\psfrag{f3}{$14$}

\psfrag{f4}{$34$}

\psfrag{f5}{$25$}

\psfrag{f6}{$46$}

\psfrag{f7}{$27$}

\psfrag{f8}{$28$}

\psfrag{f9}{$25$}

\psfrag{f10}{$33$}

\psfrag{f11}{$27$}

\psfrag{f12}{$27$}

\psfrag{f13}{$27$}

\psfrag{f14}{$28$}

\psfrag{f15}{$\phantom{11}$}

\psfrag{f16}{$36$}

\psfrag{f17}{$36$}

\psfrag{f18}{$36$}

\psfrag{f19}{$56$}

\psfrag{f20}{$\phantom{11}$}

\psfrag{f21}{$55$}

\psfrag{f22}{$55$}

\psfrag{f23}{$56$}

\psfrag{f24}{$56$}

\psfrag{f25}{$\phantom{11}$}

\psfrag{f1g}{$16$}

\psfrag{f2g}{$34$}

\psfrag{f3g}{$33$}

\psfrag{f4g}{$15$}

\psfrag{f5g}{$55$}

\psfrag{f6g}{$17$}

\psfrag{f7g}{$17$}

\psfrag{f8g}{$28$}

\psfrag{f9g}{$22$}

\psfrag{f10g}{$25$}

\psfrag{f11g}{$25$}

\psfrag{f12g}{$55$}

\psfrag{f13g}{$46$}

\psfrag{f14g}{$27$}

\psfrag{f15g}{$\phantom{11}$}

\psfrag{f16g}{$37$}

\psfrag{f17g}{$34$}

\psfrag{f18g}{$45$}

\psfrag{f19g}{$36$}

\psfrag{f20g}{$\phantom{11}$}

\psfrag{f21g}{$45$}

\psfrag{f22g}{$55$}

\psfrag{f23g}{$56$}

\psfrag{f24g}{$56$}

\psfrag{f25g}{$\phantom{11}$}

\begin{center}
\epsfig{file=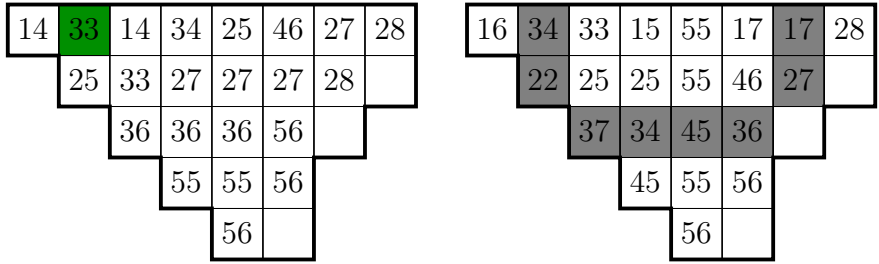,height=3.5cm}
\end{center}
\caption{A representative of the left-hand side and the right-hand side of \eqref{sh-branch} for $\lambda = 87532$.}
\label{fig7}
\end{figure}

\medskip

Similarly, the left-hand side of \eqref{sh-branch} counts the number of pairs $(\s c, G)$, where $\s c \in \p C^*[\lambda]$ and $G$ is a map from $[\lambda]^* \setminus \p C^*[\lambda]$ to $[\lambda]^*$ satisfying the following:
\begin{itemize}
 \item if $\s c$ lies in the hook of $\s z$, then $G(\s z)$ lies in the hook of $\s z$;
 \item otherwise, $G(\s z)$ lies in the punctured hook of $\s z$.
\end{itemize}
Note that if $\s c = (r,s)$, then $\s c$ lies in the hook of $\s z = (i,j)$ if and only if $i = r$ or $j = r-1$ or $j = s$. The right drawing in Figure \ref{fig7} shows an example of such a pair $(\s c, G)$; the row $r$ without the corners $(r,s)$ and the columns $r-1$ and $s$ without the corner $(r,s)$ are shaded. Denote the set of all $G$ such that $(\s c,G)$ satisfies the above properties by $\p G_{\s c}$, and the disjoint union $\bigsqcup_{\s c} \p G_{\s c}$ by $\p G$. We also write $G_{\s z} = G_{i,j}$ for $G(\s z)$ when $\s z = (i,j)$. We think of the map $G$ as an arrangement of labels.

\medskip

Some squares $\s z$ in shaded row and columns may satisfy $G_{\s z} = \s z$. Let us define sets $A,B,C$ as follows:
\begin{eqnarray*}
 \aligned
 A = A(G) & = \set{i \colon G_{i,s} = (i,s)} \\
 B = B(G) & = \set{j \colon G_{r,j} = (r,j)}\\
 C = C(G) & = \set{i \colon G_{i,r-1} = (i,r-1)} 
 \endaligned
\end{eqnarray*}

We think of sets $A$, $B$ and $C$ as representing column $s$ without $(r,s)$, row $r$ without $(r,s)$, and column $r-1$ respectively, with a dot in all squares satisfying $G_{\s z} = \s z$. Figure \ref{fig8} shows shaded row and columns for the example from the right-hand side of Figure \ref{fig7}.  The square $s(A,B,C)$, which will be defined in the next section, is green.

\begin{figure}[hbt]

\begin{center}
\epsfig{file=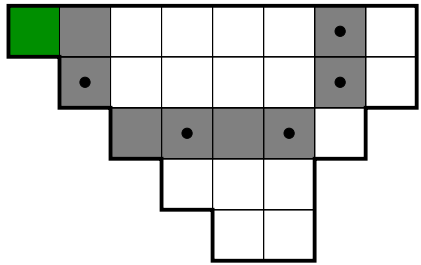,height=3.5cm}
\end{center}
\caption{Shaded row and columns and sets $A$, $B$ and $C$ for the example on the right-hand side of Figure \ref{fig7}, and $s(A,B,C)$.}
\label{fig8}
\end{figure}

\medskip

Furthermore, denote by $\p G_{\s c}^k$ the set of all $G \in \p G_{\s c}$ for which exactly $k$ of the squares in $[\lambda]^* \setminus \p C^*[\lambda]$ satisfy $G_{\s z} = \s z$ (note that such $\s z$ must be in shaded row and columns); in other words, $|A| + |B| + |C| = k$. 

\medskip

Before we continue, let us introduce some terminology relating to the shaded columns (the shaded row will not be relevant for these purposes). If $C$ is empty, we say that the shaded columns form a \emph{right stick}. If $A$ is empty and $C$ is non-empty, we say that the shaded columns form a \emph{left stick}. If $A = C$ and $|A| = |C| \geq 1$, we say that the shaded columns form a \emph{block}. If $A$ and $C$ are non-empty and $\min A = \max C$, we say that the shaded columns form a \emph{right snake}. If $A$ and $C$ are non-empty and $\max A = \min C$, we say that the shaded columns form a \emph{left snake}. The origin of these names should be clear from examples in Figure \ref{fig9}. Formations such as snake-dot, stick-dot, snake-block, snake-block-dot should be self-explanatory; see Figure \ref{fig9} for examples. We will often say that the shaded columns \emph{start} with a certain formation, for example, that the shaded columns start with a snake-dot. This means that if we erase dots in rows from some row $i$ onward, the resulting columns have the specified form. See Figure \ref{fig9} for examples.

\begin{figure}[hbt]

\begin{center}
\epsfig{file=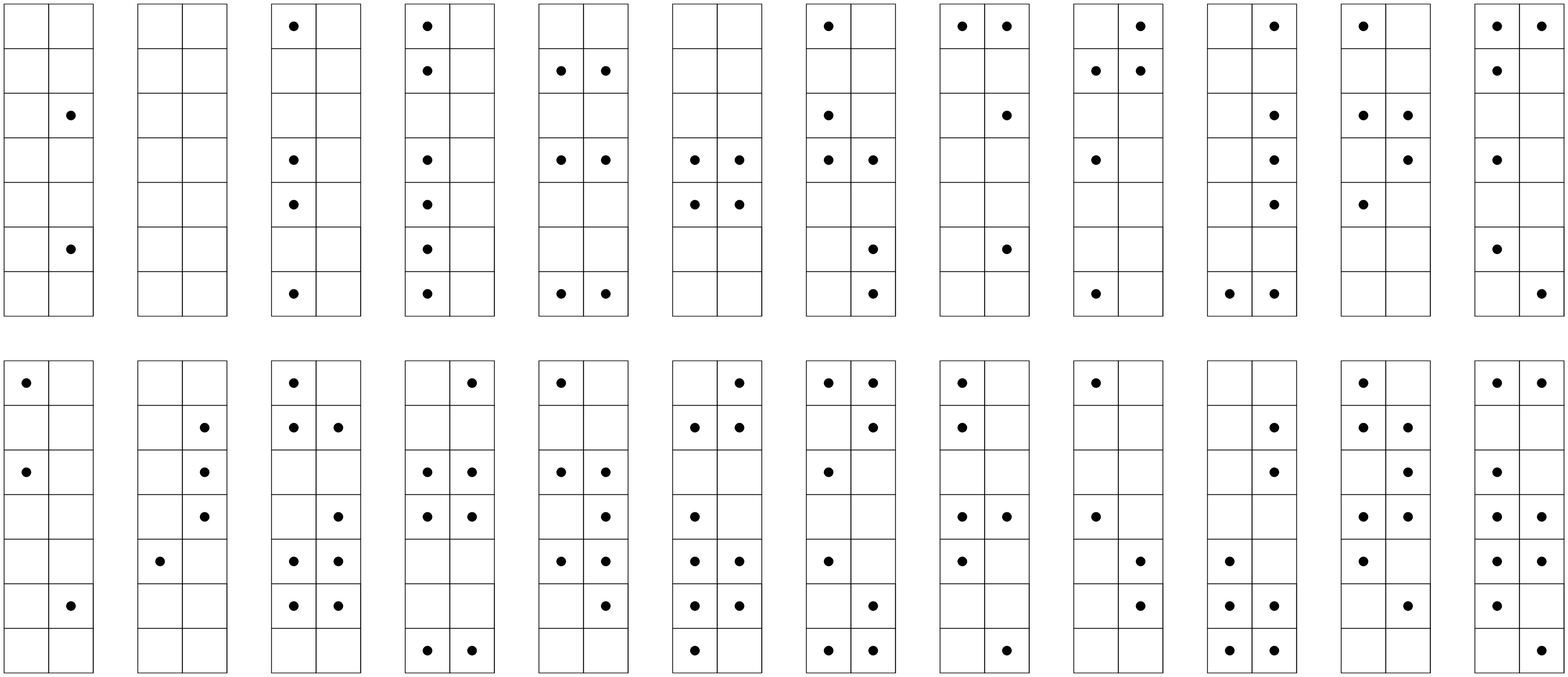,height=5cm}
\end{center}
\caption{In this figure, we have two examples of each of the following: the shaded columns form a right stick, left stick, block, right snake, left snake, snake-dot (top row), stick-dot, snake-block, snake-block-dot; the shaded columns start with a snake-dot, stick-dot, snake-block-dot (bottom row).}
\label{fig9}
\end{figure}

\medskip

Like in the non-shifted case, we will construct two maps, $s = s_{\s c}$ and $e = e_{\s c}$. The map $s \colon \p P([r-1]) \times \p P([r,\ldots,s-1]) \times \p P([r-1]) \to [\lambda]$ (``starting square''), which we extend to $s \colon \p G_{\s c} \to [\lambda]$ by $s(G) = s(A(G),B(G),C(G))$, and $e \colon \p G_{\s c} \to \p G_{\s c}$ (``erase a dot''). Again, they will have the properties (P\ref{p1}), (P\ref{p2}) and (P\ref{p3}) from Section \ref{non-shifted}. By the same argument as in the previous section, this implies that $G \mapsto (s(G),\phi(G))$, $\phi(G) = e(e(\cdots(e(G))\cdots))$, is a bijection from $\p G$ to $\p F$, and proves \eqref{sh-branch}.

\medskip

While this all seems almost exactly the same as in the non-shifted case, there are a few important differences. First, the map $s$ is fairly complicated; we have to split the possible configurations of the sets $A$, $B$ and $C$ into five separate cases, with the rule for finding the starting column fairly unintuitive. Second, in certain cases it is not immediately obvious which of the dots in the shaded columns should be erased. Note that a square $(i,j)$ with $j < r-1$ has two ``projections'' in column $r-1$ and two in column $s$, see Figure \ref{fig18}. 

\begin{figure}[hbt]
\psfrag{G}{$G$}
\begin{center}
\epsfig{file=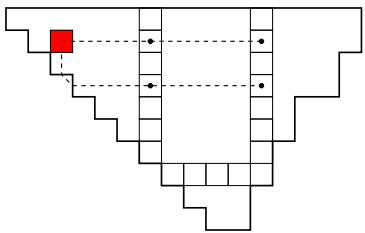,height=3cm}
\end{center}
\caption{Two projections of $(i,j)$ in column $r-1$, $(i,r-1)$ and $(j+1,r-1)$, and two projections in columns $s$, $(i,s)$ and $(j+1,s)$.}
\label{fig18}
\end{figure}

\medskip

On a related note, while the equality $h_{i,j} - 1 = (h_{i,s} - 1) + (h_{r,j} - 1)$ translates nicely into equalities $h_{i,j}^* - 1 = (h_{i,s}^* - 1) + (h_{j+1,r-1}^* - 1) = (h_{i,r-1}^* - 1) + (h_{j+1,s}^* - 1)$, it is far less obvious how to decompose the punctured hook of $(i,j)$ into punctured hooks of $(i,s)$ and $(j+1,r-1)$ (or $(i,r-1)$ and $(j+1,s)$). And the final, and perhaps the most intriguing, difference is that in most cases when the shaded columns start with a snake, the map $e$ does not just ``erase the dot'', but also ``flips the snake''. See Figure \ref{fig10} for some possible effects of $e$.

\begin{figure}[hbt]
\psfrag{G}{$G$}
\psfrag{eG}{$e(G)$}
\begin{center}
\epsfig{file=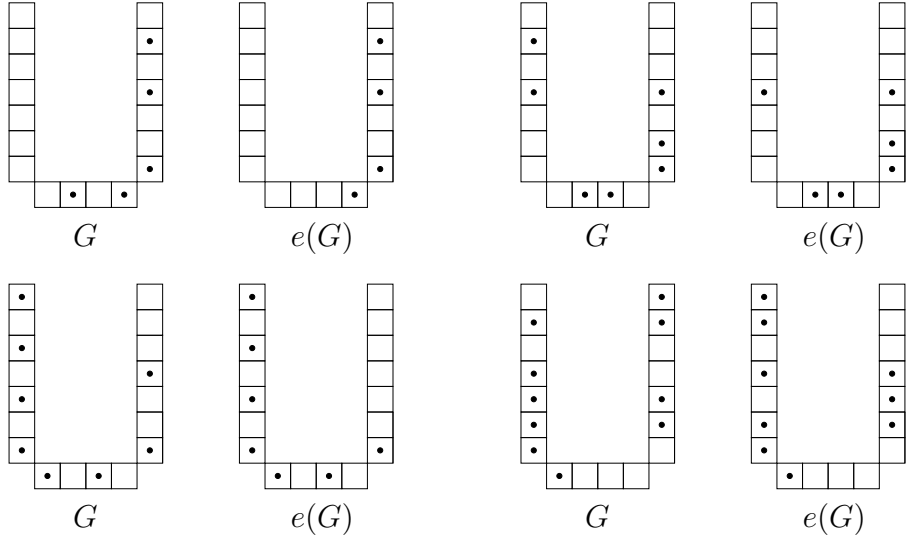,height=7cm}
\end{center}
\caption{Four possible pairs of $A(G)$, $B(G)$, $C(G)$ and $A(e(G)$, $B(e(G)$, $C(e(G))$. Note that in the last example, the right snake becomes a left snake.}
\label{fig10}
\end{figure}

\medskip

These issues are addressed in the next section. But first, we describe how to change the direction of a snake.

\medskip

We need some new notation. Pick a set $I = \set{i_1,i_2,\ldots,i_m} \subseteq [r-1]$. For $1 \leq k \leq m$, denote by $\p L_{I.k}$ the set of all arrangements $G \in \p G$ for which $i_1,\ldots,i_k \in A(G)$, $i_{k+1},\ldots,i_m \notin A(G)$, $i_1,\ldots,i_{k-1} \notin C(G)$, $i_k,\ldots,i_m \in C(G)$, and by $\p L_I$ the set $\bigcup_k L_{I,k}$. For $1 \leq k \leq m$, denote by $\p R_{I.k}$ the set of all arrangements $G \in \p G$ for which $i_1,\ldots,i_k \in C(G)$, $i_{k+1},\ldots,i_m \notin C(G)$, $i_1,\ldots,i_{k-1} \notin A(G)$, $i_k,\ldots,i_m \in A(G)$, and by $\p R_I$ the set $\bigcup_k R_{I,k}$. 

\medskip

The set $\p L_I$ (respectively, $\p R_I$) contains all arrangements so that the dots in rows $I$ form a left snake (respectively, right snake). Figure \ref{fig16} shows two examples. 

\begin{figure}[hbt]
\psfrag{f1g}{$23$}

\psfrag{f2g}{$18$}

\psfrag{f3g}{$18$}

\psfrag{f4g}{$57$}

\psfrag{f5g}{$15$}

\psfrag{f6g}{$19$}

\psfrag{f7g}{$18$}

\psfrag{f8g}{$19$}

\psfrag{f9g}{$39$}

\psfrag{f10g}{$\phantom{11}$}

\psfrag{f11g}{$29$}

\psfrag{f12g}{$48$}

\psfrag{f13g}{$44$}

\psfrag{f14g}{$28$}

\psfrag{f15g}{$66$}

\psfrag{f16g}{$77$}

\psfrag{f17g}{$28$}

\psfrag{f18g}{$39$}

\psfrag{f19g}{$34$}

\psfrag{f20g}{$38$}

\psfrag{f21g}{$35$}

\psfrag{f22g}{$37$}

\psfrag{f23g}{$39$}

\psfrag{f24g}{$38$}

\psfrag{f25g}{$\phantom{11}$}

\psfrag{f26g}{$57$}

\psfrag{f27g}{$45$}

\psfrag{f28g}{$48$}

\psfrag{f29g}{$67$}

\psfrag{f30g}{$48$}

\psfrag{f31g}{$55$}

\psfrag{f32g}{$77$}

\psfrag{f33g}{$67$}

\psfrag{f34g}{$68$}

\psfrag{f35g}{$68$}

\psfrag{f36g}{$67$}

\psfrag{f37g}{$\phantom{11}$}

\psfrag{f38g}{$\phantom{11}$}

\psfrag{f14}{$25$}

\psfrag{f17}{$68$}

\psfrag{f21}{$35$}

\psfrag{f24}{$38$}

\psfrag{f31}{$58$}

\psfrag{f34}{$58$}

\begin{center}
\epsfig{file=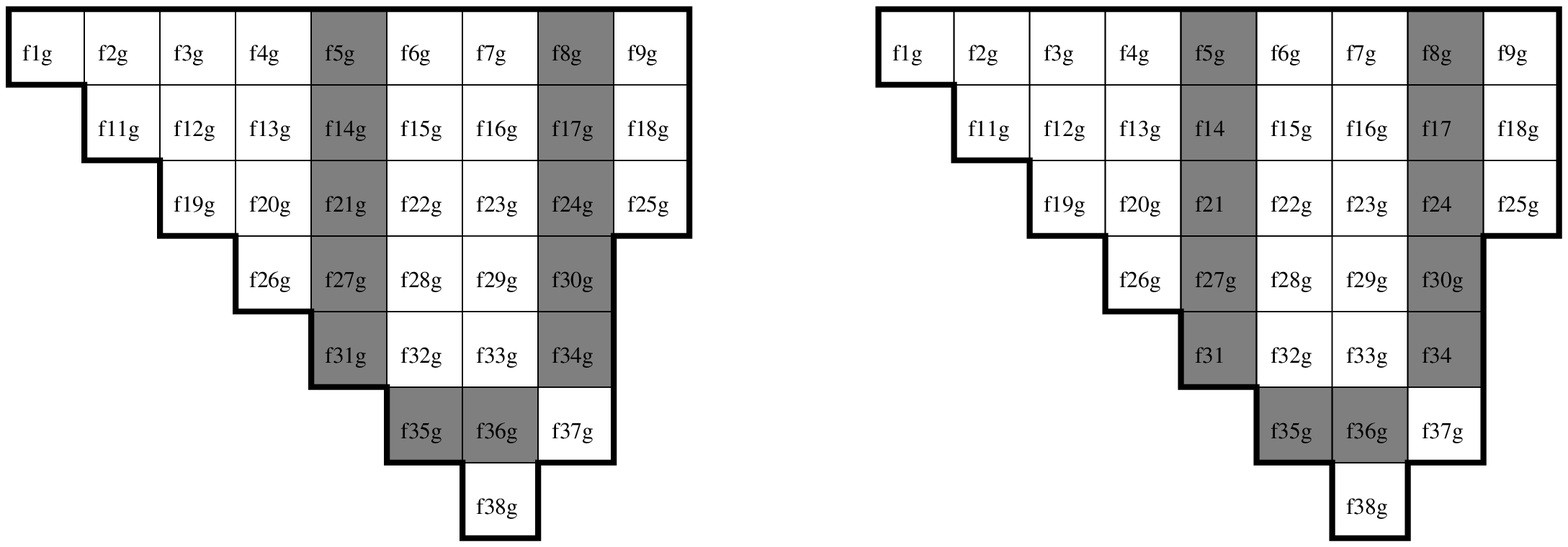,height=4.9cm}
\end{center}
\caption{An element of $\p L_I$ and $\p R_I$ for $\lambda = 9875431$, $\s c = (6,8)$, $I = \set{2,3,5}$.}
\label{fig16}
\end{figure}

\begin{lemma}[Flipping the snake] \label{flip}
 There is a bijection $\Psi_I \colon \p L_I \to \p R_I$ with the property that $G_{\s z} = \Psi_I(G)_{\s z}$ if $\s z \notin I \times \set{r-1,s}$. 
\end{lemma}
\begin{proof}
 We construct the map $\Psi_I$ inductively. If $|I| = 0$ or $|I| = 1$, the sets $\p L_I$ and $\p R_I$ are equal and we can take $\Psi_I$ to be the identity map. Now take $I = \set{i_1,i_2,\ldots,i_m}$, $m \geq 2$, $I' = I \setminus \set{i_1} = \set{i_2,\ldots,i_m}$, $G \in \p L_{I.k}$, and look at $G_{i_1,r-1}$.\\
 First take $k = 1$. Then $G_{i_1,r-1} = (i_1,r-1)$, $G_{i_1,s} = (i_1,s)$, $G_{i_2,r-1} = (i_2,r-1)$, $G_{i_2,s} \in \overline H^*_{i_2,s}$. A square $(i,s)$, $i > i_2$, in $\overline H^*_{i_2,s}$ is also in the punctured hook of $(i_1,s)$, and a square $(i_2,j)$, $j > s$, of $\overline H^*_{i_2,s}$ is in a natural correspondence with the square $(i_1,j)$, which is in the punctured hook of $(i_1,s)$. By slight abuse of notation, we can therefore define $G'_{i_1,s} = G_{i_2,s}$; we also define $G'_{i_2,s} = (i_2,s)$ and $G'_{\s z} = G_{\s z}$ for $\s z \neq (i_1,s), (i_2,s)$.\\
 Now note that $G'$ is an element of $\p L_{I'}$, and therefore we can take $G'' = \Psi_{I'}(G') \in \p R_{\p I'}$. By assumption, $\Phi_{I'}$ does not change values in row $i_1$, so $G''_{i_1,r} = G'_{i_1,r} = G_{i_1,r} = (i_1,r)$ and $G''_{i_1,s} = G'_{i_1,s} \in \overline H^*_{i_1,s}$. That means that $G'' \in \p R_I$, and we take $\Phi_I(G) = G''$.\\
 Now assume $k > 1$. If $G_{i_1,r-1}$ is of the form $(i_1,j)$ for $j > s$, then it is also in the punctured hook of $(i_1,s)$, and if it is of the form $(i,r-1)$ for $i \leq i_l$, then $(i,s)$ is in the punctured hook of $(i_1,s)$. By slight abuse of notation, we can then take $G'_{i_1,s} = G_{i_1,r-1}$, $G'_{i_1,r-1} = (i_1,r-1)$, $G'_{\s z} = G_{\s z}$ for $\s z \neq (i_1,r-1), (i_1,s)$. Since $G' \in \p L_{I'}$, we can take $G'' = \Phi_{I'}(G')$, and since $G'' \in \p R_I$, we define $\Phi_{I}(G) = G''$.\\
 Otherwise, $G_{i_1,r-1}$ is in the punctured hook of $(i_k,r-1)$ (where we identify squares $(i_1,j)$ and $(i_k,j)$ for $r-1 < j \leq s$). Note that if $k < m$, the square $G_{i_{k+1},s}$ is in the punctured hook of $(i_1,s)$ (where we identify squares $(i_1,j)$ and $(i_{k+1},j)$ for $s < j$). Define $G'_{i_1,r-1} = (i_1,r-1)$, $G'_{i_1,s} = G_{i_{k+1},s}$ (if $k < m$), $G'_{i_1,s} = G_{i_1,s}=(i_1,s)$ (if $k = m$), $G'_{i_k,r-1} = G_{i_1,r-1}$, $G'_{i_{k+1},s} = (i_{k+1},s)$ (if $k < m$), $G'_{\s z} = G_{\s z}$ for all other $\s z$. If $k = m$, then $G' \in \p R_I$, take $G'' = G'$; otherwise, $G' \in \p L_{I'}$, so we can define $G'' = \Phi_{I'}(G')$. Since $G'' \in \p R_I$ in either case, we define $\Phi_{I}(G) = G''$.\\
 It is easy to construct the analogous map $\p R_I \to \p L_I$, and to prove that such maps are inverses. The details are left to the reader.
\end{proof}

\begin{exm}
 The arrangement on the right-hand side of Figure \ref{fig16} is the image of the arrangement on the left-hand side.\\
 The following figure shows four further examples. Each block shows shaded columns for $G$, $G'$ and $G''$. In the first example, we have $G \in \p L_{I,1}$. In the second and third example, we have $G \in \p L_{I,l}$ with $1 < k < m$; $G'$ can have different $A$ and $C$, depending on $G_{i_1,s}$. The last example shows what happens if $G \in \p L_{I,m}$ and $G_{i_1,r-1}$ is in the punctured hook of $(i_m,r-1)$.

\begin{figure}[hbt]

\begin{center}
\epsfig{file=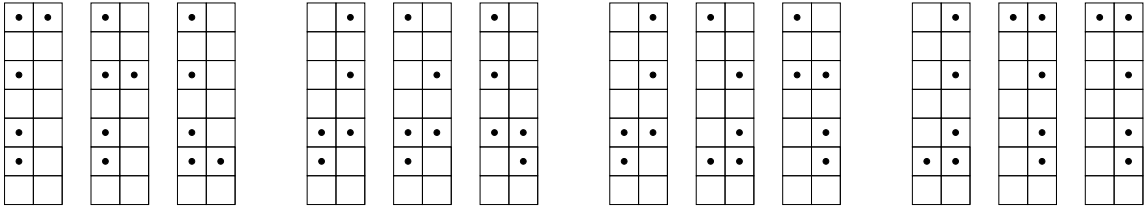,height=2.7cm}
\end{center}
\caption{Left snake, intermediate step, right snake, for four examples, with $r = 8$ and $I = \set{1,3,5,6}$.}
\label{fig14}
\end{figure}

\end{exm}

\section{Detailed description of the maps $s$ and $e$} \label{details}

In this section, we describe in full detail the two maps mentioned in Section \ref{basic}, $s$ and $e$. While we call the map $e$ ``erasing the dot'', its effect on the shaded row and columns is, as mentioned above, a bit more complicated.

\subsection*{Choosing the starting square}

We are given $G \in \p G_{\s c}$, and we want to construct the square $\s z = (i,j) \in [\lambda]^*$. For the row, the rule is simple and is a straightforward generalization of the rule in the non-shifted case.

\begin{enumerate}
 \renewcommand{\labelenumi}{(S\arabic{enumi})}
 \setcounter{enumi}{-1}
 \item \label{s0} $i$ is the row coordinate of the upper-most dot in the shaded columns; if both $A$ and $C$ are empty, take $i = r$. In other words, $i = \min(A \cup C \cup \set r)$.
\end{enumerate}

For the column $j$, there are several different rules depending on the formation of the shaded columns. The motivation for these rules is as follows. We want to think of $s(A,B,C)$ as the starting square of the hook walk that corresponds to the arrangement $G$. If $C$ is empty, we are essentially in the situation of the non-shifted case, and $s(A,B,C)$ is defined accordingly. If the shaded columns form a left stick or a right snake and $B$ is empty, then the dots determine a hook walk, which, of course, starts in column $r-1$. In other cases, we start from the top of the shaded columns, and try to see for how long the dots form a hook walk, i.e., a snake (either left or right). Once we hit an obstruction (like a block or a dot in the wrong column), we take the column corresponding to the row where the obstruction occurs, or to the last snake row.

\medskip

More precisely, we have exactly one of the following options:
\begin{enumerate}
 \renewcommand{\labelenumi}{(S\arabic{enumi})}
 \item \label{s1} the shaded columns form a right stick; in this case, $j$ is the column coordinate of the left-most dot in the shaded row or, if $B$ is empty, $j = s$;
 \item \label{s2} the shaded columns form a left stick or a right snake; in this case, $j = r - 1$;
 \item \label{s3} the shaded columns start with a stick-dot; in this case $j = k - 1$, where $k$ is the row coordinate of the dot;
 \item \label{s4} the shaded columns start with a block of length $\geq 2$; in this case $j = k-1$, where $k$ is the row coordinate of the second block row;
 \item \label{s5} the shaded columns start with a snake and do not form a right snake; in this case, count the number of true statements in the following list:
 \begin{itemize}
  \item the shaded columns start with a left snake;
  \item the shaded columns start with a snake-odd block;
  \item the shaded columns start with a snake-left dot or snake-block-left dot.
 \end{itemize}
 If this number is odd, let $j = k-1$, where $k$ is the row coordinate of the last snake row. If this number is even, the shaded columns must start with a snake-block; let $j = k-1$, where $k$ is the row coordinate of the first block row.
\end{enumerate}

An alternative description of (S\ref{s5}) is as follows. If the shaded columns form a left snake or a left snake snake-even block or a right snake-odd block or start with a snake-dot or a left snake-even block-right dot or a left snake-odd block-left dot or a right snake-even block-left dot or a right snake-odd block-right dot, let $j = k-1$, where $k$ is the row coordinate of the last snake row. If the shaded columns form a left snake-odd block or right snake-even block or start with a left snake-even block-left dot or a left snake-odd block-right dot or a right snake-even block-right dot or a right snake-odd block-left dot, let $j = k-1$, where $k$ is the row coordinate of the first block row. And if the shaded columns start with a snake-dot, let $j = k-1$, where $k$ is the row coordinate of the first block row.

\medskip

Figure \ref{fig11} should illuminate these rules. Clearly, (S\ref{s0}) and (S\ref{s1}) imply the first statement of (P\ref{p1}).

\begin{figure}[hbt]
\psfrag{G}{$G$}
\begin{center}
\epsfig{file=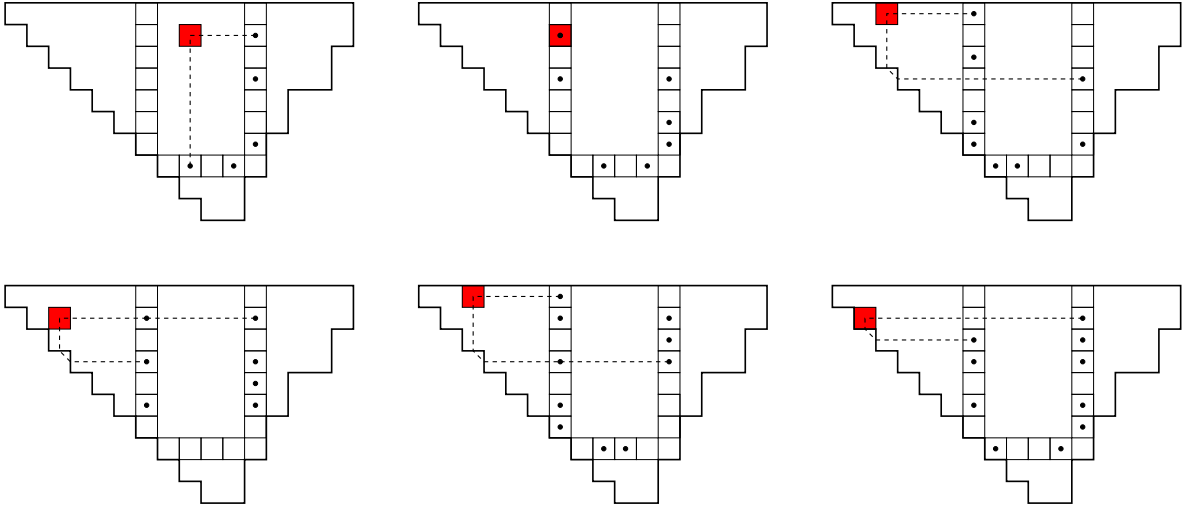,height=6.7cm}
\end{center}
\caption{The starting square when the shaded columns form a right stick, left snake, or start with a stick-dot (top), block, snake-block-dot -- two examples (bottom).}
\label{fig11}
\end{figure}

\medskip

In all cases, we could give these rules in terms of the sets $A$ and $C$ (sets $A$ and $B$ in (S\ref{s1})). For example, we could state (S\ref{s3}) as follows: if $C \neq \emptyset$ and $A \cap [\min C] \neq \emptyset$ and $\max(A \cap [\min C]) < \min C$, then $j = \min C - 1$, and if $A \neq \emptyset$ and $C \cap [\min A] \neq \emptyset$ and $\max(C \cap [\min A]) < \min A$, then $j = \min A - 1$. Such statements seem to be, however, much less intuitive than the descriptions above. 

\subsection*{Erasing the dot} In this subsection, we construct the map $e$, which acts as identity on $\p G_{\s c}^0$, and which maps $\p G_{\s c}^k$ into $\p G_{\s c}^{k-1}$ for $k \geq 1$; in other words, if there is at least one dot in the shaded columns and row, then $e(G)$ should have one fewer dot in these columns and row. Furthermore, $s$ and $e$ should satisfy properties (P\ref{p1}), (P\ref{p2}) and (P\ref{p3}).

\medskip

So pick $G \in \p G_{\s c} \setminus \p G_{\s c}^0$. There are three possibilities. First, $s(G) = (i,j)$ can be one of the shaded squares (i.e., $i = r$ or $j = r - 1$ or $j = s$); this happens in case (S\ref{s1}) if either $A = \emptyset$ or $B = \emptyset$, and in case (S\ref{s2}), and $e(G)$ differs from $G$ in exactly one square. Second, $s(G)$ can be a non-shaded square that appears to the right of the column $C$ (i.e., we have $i < r$ and $r - 1 < j < s$); this happens in case (S\ref{s1}) if both $A$ and $B$ are non-empty, and $e(G)$ differs from $G$ in exactly two squares. And third, $s(G)$ can be a non-shaded square to the left of the column $C$ (i.e., $j < r - 1$); this happens in cases (S\ref{s3}), (S\ref{s4}) and (S\ref{s5}). Now, $e(G)$ differs from $G$ in at least two squares.

\medskip

In the first case, $e(G)$ differs from $G$ only at $s(G)$. If $B = C = \emptyset$ and $A = \set{i,\ldots}$, then $s(G) = (i,s)$ and $G_{i,s} = (i,s)$; define $A' = A \setminus \set i$, $B' = B$, $C' = C$, and $e(G)_{i,s} = s(A',B',C')$. If $A = C = \emptyset$ and $B = \set{j,\ldots}$, then $s(G) = (r,j)$ and $G_{r,j} = (r,j)$, define $A' = A$, $B' = B \setminus \set j$, $C' = C$, and $e(G)_{r,j} = s(A',B',C')$. And otherwise (if $A$ and $C = \set{i,\ldots}$ form a left stick or a right snake), we have $s(G) = (i,r-1)$ and $G_{i,r-1} = (i,r-1)$; define $A' = A$, $B' = B$, $C' = C \setminus \set i$, and $e(G)_{i,r-1} = s(A',B',C')$. Then define $e(G)_{\s z} = G_{\s z}$ for $\s z \neq s(G)$. In all these cases, $A(e(G)) = A'$, $B(e(G))=B'$, $C(e(G)) = C'$ and $e(G)_{s(G)} = s(e(G))$. 

\medskip

The second case, when $A = \set{i,\ldots}$, $B = \set{j,\ldots}$, $C = \emptyset$, is almost exactly the same as in the non-shifted case. The square $G_{s(G)}$ is in the punctured hook of $s(G) = (i,j)$. This punctured hook is in a natural correspondence with $\overline H_{is}^* \sqcup \overline H_{rj}^*$. See Figure \ref{fig12}. So by slight abuse of notation, we define $e(G)$ as follows. If $G_{i,j} \in \overline H_{is}^*$, define $e(G)_{i,s} = G_{i,j}$, and if $G_{s(G)} \in \overline H_{rj}^*$, define $e(G)_{r,j} = G_{i,j}$. In the first case, also define $A' = A \setminus \set i$, $B' = B$, $C' = C$, and in the second case, define $A' = A$, $B' = B \setminus \set j$, $C' = C$. Then take $e(G)_{i,j} = s(A',B',C')$, $e(G)_{\s z} = G_{\s z}$ for $\s z \neq s(G)$. In both cases, we have $A(e(G)) = A'$, $B(e(G))=B'$, $C(e(G)) = C'$ and $e(G)_{s(G)} = s(e(G))$.

\begin{figure}[hbt]
\psfrag{G}{$G$}
\begin{center}
\epsfig{file=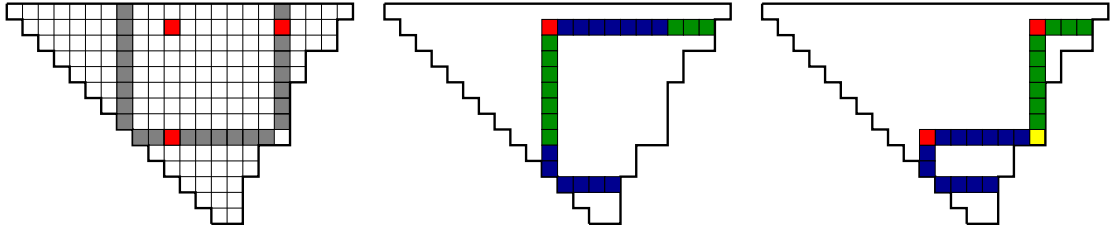,height=3cm}
\end{center}
\caption{The punctured hook of a square is the disjoint union of the punctured hooks of projections onto shaded row and column. The yellow square is in the punctured hook of both projections.}
\label{fig12}
\end{figure}

\medskip

As mentioned before, the third case splits into three possible subcases. the shaded columns either start with a stick-dot, or with a block of length $\geq 2$, or start with a snake and do not form a right snake. Again, it is important to decompose the punctured hook of $s(G) = (i,j)$, where $j < r - 1$. The important difference now is that we have two possible decompositions; we can either decompose it as $\overline H_{is}^* \sqcup \overline H_{j+1,r-1}^*$, or as $\overline H_{i,r-1}^* \sqcup \overline H_{j+1,s}^*$. See Figure \ref{fig13}.

\begin{figure}[hbt]
\psfrag{G}{$G$}
\begin{center}
\epsfig{file=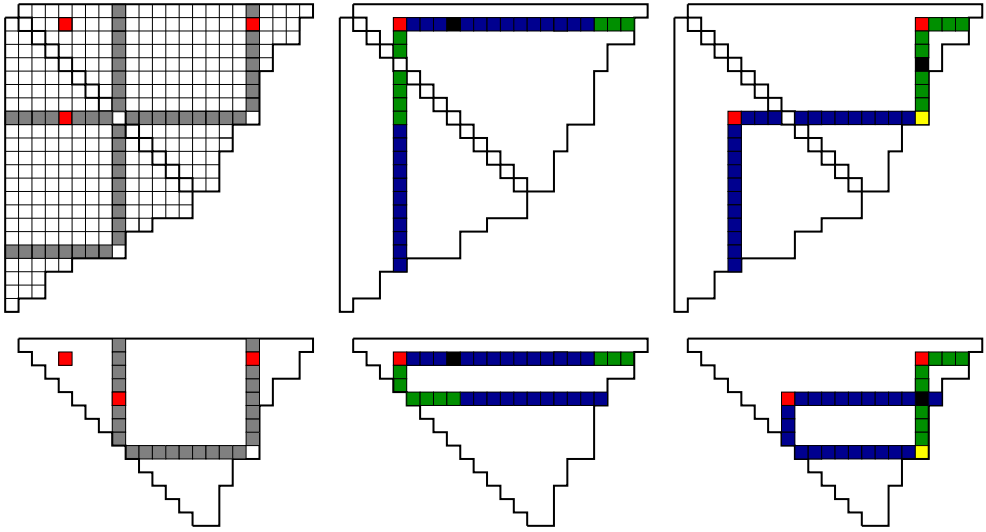,height=7cm}
\end{center}
\caption{The punctured hook of a square is the disjoint union of the punctured hooks of projections onto shaded columns. The decomposition is shown on the bottom, and the top is a graphic explanation of why this decomposition works. The yellow square is in the punctured hook of both projections, as is the black square in the top right picture.}

\label{fig13}
\end{figure}

\medskip

In the cases (S\ref{s3}), (S\ref{s4}) and (S\ref{s5}), we therefore describe the following:
\begin{itemize}
 \item which decomposition, $\overline H_{is}^* \sqcup \overline H_{j+1,r-1}^*$ or $\overline H_{i,r-1}^* \sqcup \overline H_{j+1,s}^*$, we use; 
 \item whether or not we have to use Lemma \ref{flip} to ``flip the snake'' (this can only happen in case (S\ref{s5})).
\end{itemize}

If we use decomposition $\overline H_{is}^* \sqcup \overline H_{j+1,r-1}^*$ and $G_{i,j} \in \overline H_{is}^*$, then define $e(G)_{i,s} = G_{i,j}$, $e(G)_{\s z} = G_{\s z}$ for $\s z \neq (i,s), (i,j)$, $A' = A \setminus \set i$, $B' = B$, $C' = C$. If $G_{i,j} \in \overline H_{j+1,r-1}^*$, then define $e(G)_{j+1,r-1} = G_{i,j}$, $e(G)_{\s z} = G_{\s z}$ for $\s z \neq (j+1,r-1), (i,j)$, $A' = A$, $B' = B$, $C' = C \setminus \set{j+1}$. Similarly, if we use decomposition $\overline H_{i,r-1}^* \sqcup \overline H_{j+1,s}^*$ and $G_{i,j} \in \overline H_{i,r-1}^*$, then define $e(G)_{i,r-1} = G_{i,j}$, $e(G)_{\s z} = G_{\s z}$ for $\s z \neq (i,r-1), (i,j)$, $A' = A$, $B' = B$, $C' = C \setminus \set i$. And if $G_{i,j} \in \overline H_{j+1,s}^*$, then define $e(G)_{j+1,s} = G_{i,j}$, $e(G)_{\s z} = G_{\s z}$ for $\s z \neq (j+1,s), (i,j)$, $A' = A \setminus \set{j+1}$, $B' = B$, $C' = C$. 

\medskip

As we will see, in cases (S\ref{s3}), (S\ref{s4}) and in one subcase of (S\ref{s5}), $s(A',B',C')$ is always in the punctured hook of $s(G) = s(A,B,C)$. In this case, define $e(G)_{i,j} = s(A',B',C')$.

\medskip

In most subcases of (S\ref{s5}), $s(A',B',C')$ is not in the punctured hook of $s(G) = s(A,B,C)$. In this case, we will use Lemma \ref{flip} to get $A''$, $B''$ and $C''$ with $|A''| + |B''| + |C''| = |A'| + |B'| + |C'|$ and so that $s(A'',B'',C'')$ is in the punctured hook of $s(G) = s(A,B,C)$. In this case, define $e(G)_{i,j} = s(A'',B'',C'')$.

\medskip

We can only use decomposition $\overline H_{is}^* \sqcup \overline H_{j+1,r-1}^*$ if $i \in A$ and $j+1 \in C$, and we can only use decomposition $\overline H_{i,r-1}^* \sqcup \overline H_{j+1,s}^*$ if $i \in C$ and $j+1 \in A$.

\medskip

Accordingly, for (S\ref{s3}), we have only one choice for the decomposition: if the shaded columns start with a right stick-left dot, use decomposition $\overline H_{is}^* \sqcup \overline H_{j+1,r-1}^*$, and if they start with a left stick-right dot, use decomposition $\overline H_{i,r-1}^* \sqcup \overline H_{j+1,s}^*$. In both cases $s(A',B',C')$ is in the punctured hook of $s(A,B,C)$: let us prove this only when we have a right stick-left dot. Suppose first that $A = \set{i,i',\ldots,}$, $C = \set{j+1,j'+1,\ldots}$ for $i' < j+1$. If $G_{i,j} \in \overline H_{is}^*$, then $A' = \set{i',\ldots}$ and $C' = \set{j+1,j'+1,\ldots}$ also form a stick-dot, and $s(A',B',C') = (i',j)$, which is in the punctured hook of $(i,j)$. If $G_{i,j} \in \overline H_{j+1,r-1}^*$, then the starting row for $A'$, $B'$ and $C'$ is still $i$, but the starting column is at least $j'$; thus $s(A',B',C')$ is in the punctured hook of $(i,j)$. Now suppose that $A = \set{i,i',\ldots,}$, $C = \set{j+1,j'+1,\ldots}$ for $i' > j+1$. If $G_{i,j} \in \overline H_{is}^*$, then the first coordinate of $s(A',B',C')$ for $A' = \set{i',\ldots}$ and $C' = \set{j+1,j'+1,\ldots}$ is $j+1$, and all of row $j+1$ is in the punctured hook of $(i,j)$. And if $G_{i,j} \in \overline H_{j+1,r-1}^*$, then the starting row for $A'$, $B'$ and $C'$ is still $i$, but the starting column is at least $j'$; thus $s(A',B',C')$ is in the punctured hook of $(i,j)$. The case when $|A| = 1$ or $|C| = 1$ is easy.

\medskip

Take case (S\ref{s4}). The shaded columns either form a block or start with a block-dot. If they form an even block, or start with an even block-right dot or odd block-left dot, use decomposition $\overline H_{is}^* \sqcup \overline H_{j+1,r-1}^*$. If they form an odd block, or start with an odd block-right dot or even block-left dot, use decomposition $\overline H_{i,r-1}^* \sqcup \overline H_{j+1,s}^*$. It is easy to see that in every case, $s(A',B',C')$ is in the punctured hook of $(i,j)$. For example, if the shaded columns start with an even block-right dot where the block has length $> 2$, $A = \set{i,j+1,j'+1,\dots}$ and $C = \set{i,j+1,j'+1,\ldots}$, then erasing $i$ from $A$ or $j+1$ from $C$ produces a formation that starts with right snake-even block-right dot; by (S\ref{s5}), $s(A',B',C') = (i,j')$, which is in the punctured hook of $(i,j)$. All other cases are tackled in a similar way.

\medskip

That leaves only the case (S\ref{s5}). The rule is as follows. If the shaded columns start with a left snake, use decomposition $\overline H_{is}^* \sqcup \overline H_{j+1,r-1}^*$, and if they start with a right snake, use decomposition $\overline H_{i,r-1}^* \sqcup \overline H_{j+1,s}^*$. Erasing $i$ from $A$ or $C$ always produces $s(A',B',C')$ which is in the punctured hook of $(i,j)$, as can be seen by an inspection of all possible cases. However, erasing $j+1$ from $A$ or $C$ has this property only if the shaded columns form a left stick-even block or right stick-odd block, or if they start with a left stick-even block-right dot, left stick-odd block-left dot, right stick-even block-left dot or right stick-odd block-right dot; see top row of Figure \ref{fig17}. Otherwise, erasing $j+1$ from $A$ or $C$ produces $A',C'$ with $s(A',B',C') = (i,j')$ for $j' \leq j$, which is \emph{not} in the punctured hook of $(i,j)$; see bottom row of Figure \ref{fig17}.

\begin{figure}[hbt]
\psfrag{G}{$G$}
\begin{center}
\epsfig{file=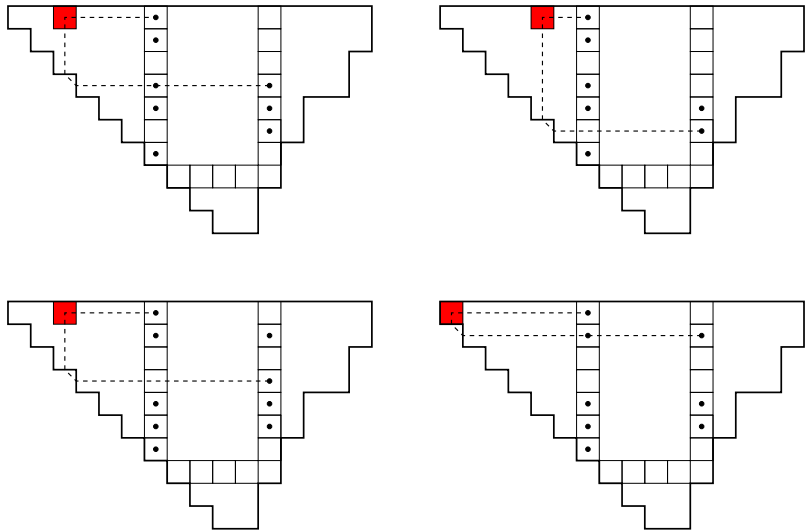,height=7cm}
\end{center}
\caption{In case (S\ref{s5}), deleting a dot in row $j+1$ can either produce $A',C'$ with $s(A',B,C')$ in the punctured hook of $s(A,B,C)$ (top) or not (bottom).}
\label{fig17}
\end{figure}

\medskip

In all these cases, however, changing a left snake to a right snake or vice versa gives $A'',B''=B,C''$ so that $s(A'',B'',C'') = (i,j')$ with $j' > j$. So Lemma \ref{flip} is the remaining ingredient of the construction of $e$. If we are in the position in case (S\ref{s5}) when erasing $j+1$ from $A$ or $C$ does not produce $A',B'=B,C'$ with $s(A',B',C')$ in the punctured hook of $s(A,B,C)$, apply $\Psi_I$ or $\Psi_I^{-1}$ to $G$ first, where $I = (A \cup C) \cap [j]$, and erase $j+1$ from the resulting $A'$ (if the original shaded columns started with a right snake) or $C'$ (if the original shaded columns started with a left snake) to get $A''$ and $C'$. This way, we get $s(A'',B'',C'')$ in the punctured hook of $s(A,B,C)$, and we are done.

\medskip

The property (P\ref{p1}) is clear, and the property (P\ref{p2}) follows by construction. Constructing the inverse implicit in (P\ref{p3}) is straightforward, if cumbersome, and is left as an exercise for the reader.

\section{Comparison with Sagan's proof} \label{comparison}

Apart from \cite{ckp} and \cite{konva}, the major source of motivation and inspiration for this paper was \cite{sagan}. Indeed, some of the notation (such as sets $A$, $B$ and $C$) was preserved on purpose. Though not explicitly stated, the map $s$ is hidden in the definition of the operators $N$, $R$, $L$, $T$ and the proof of Lemma 6. On the other hand, our snake-flipping (Lemma \ref{flip}) is equivalent to Lemma 7. There are, however, several advantages to our bijective approach. We mention a few in the following paragraphs.

\medskip

An obvious and minor difference is that in \cite{sagan}, the case $C = \emptyset$ is treated separately, in a section of its own, while we are able to place this case in a wider framework, with a special rule for $s$ and $e$. The second difference is that with our approach, there is no need to study ``basic trials'' (i.e. the case $B = \emptyset$) any differently than other cases.

\medskip

The complicated part of the proof is, of course, when $C \neq \emptyset$. Sagan offers an ad hoc construction for the case when $s(A,\emptyset,C) = (r-2,r-2)$, and the simultaneous reverse induction in the proof for $s(A,\emptyset,C) = (i,j)$ when $j < r-2$ is very tricky. Our proof, on the other hand, offers a unified description of how to select the dot to erase.

\medskip

Furthermore, we feel that the proof of the snake-flipping lemma is much more intuitive than \cite[proof of Lemma 7]{sagan}. It also sheds light on what the ``counterexample'' from \cite[\S 4]{sagan} means: that the hook walk by itself does not determine the projection onto shaded row and columns; we also need to pick random elements of the punctured hook on (some of) the squares in shaded columns. More abstractly, it says that snake-flipping is necessary; that we cannot hope to always be able to erase a dot from the shaded row and columns and hope for a new configuration with a starting square in the punctured hook of the starting square of the original configuration.

\medskip

Also, the bijective approach lends itself very naturally to the weighted formulas, and various other variants mentioned in Sections \ref{variants} and \ref{final}.

\section{Variants and weighted formulas} \label{variants}

The true power of a bijection lies in its robustness. In this section, we specialize and slightly adapt the bijection to get variants and generalizations of \eqref{sh-branch}.

\begin{thm} \label{vars}
 For a partition $\lambda$ of $n$ with distinct parts, we have the following equalities:
 \begin{equation*}
  \aligned
   n\cdot \!\!\!\!\!\!\!\!\!\prod_{\scriptscriptstyle(i,j) \in [\lambda]^* \setminus \p C^*[\lambda]} \!\!\!\!\!\!\!\!\!\left({ h^*_{ij} \!-\! 1}\right) &= \!\!\!\!\sum_{\scriptscriptstyle(r,s) \in \p C^*[\lambda]}{ \prod_{\stackrel{(i,j) \in [\lambda]^*\setminus \p C^*[\lambda]}{\scriptscriptstyle i \neq r,j \neq r-1,s}}\!\!\!\!\!\! (h_{ij}^* \!-\! 1)\prod_{i=1}^{r-1} h^*_{is} \prod_{j=r}^{s-1} h^*_{rj} \prod_{i=1}^{r-1} h^*_{i,r-1}} \\
   \lambda_1 \!\cdot \!\!\!\!\!\!\!\!\!\prod_{\scriptscriptstyle(i,j) \in [\lambda]^* \setminus \p C^*[\lambda]} \!\!\!\!\!\!\!\!\!\left({ h^*_{ij} \!-\! 1}\right) &= \!\!\!\!\sum_{\scriptscriptstyle(r,s) \in \p C^*[\lambda]}\!\!\!\!\!\left( h^*_{1,r-1} \!+\! h^*_{1s}\! - \!1 \right)\!\!\!\!\!\!{ \prod_{\stackrel{(i,j) \in [\lambda]^*\setminus \p C^*[\lambda]}{\scriptscriptstyle i \neq r,j \neq r-1,s}}\!\!\!\!\!\! (h_{ij}^* \!-\! 1)\prod_{i=2}^{r-1} h^*_{is} \prod_{j=r}^{s-1} h^*_{rj} \prod_{i=2}^{r-1} h^*_{i,r-1}} \\
   \!\!\prod_{\scriptscriptstyle(i,j) \in [\lambda]^* \setminus \p C^*[\lambda]}  \!\!\!\!\!\!\!\!\!\left({ h^*_{ij} \!-\! 1}\right) &= \!\!\!\!\sum_{\scriptscriptstyle(r,s) \in \p C^*[\lambda]}\!\!\!\!\!\left( { h^*_{1s}h^*_{2,r-1} \!+\! (h^*_{1,r-1}\! -\! 1)(h^*_{2s}\!-\!1)} \right)\!\!\!\!\!\!\!\!\!{ \prod_{\stackrel{(i,j) \in [\lambda]^*\setminus \p C^*[\lambda]}{\scriptscriptstyle i \neq r,j \neq r-1,s}}\!\!\!\!\!\!\!\! (h_{ij}^* \!- \!1)\prod_{i=3}^{r-1} h^*_{is} \!\!\!\!\!\!\prod_{j=\max\set{r,2}}^{s-1} \!\!\!\!\!\!h^*_{rj} \prod_{i=3}^{r-1} h^*_{i,r-1}}
  \endaligned
\end{equation*}
Here $h^*_{\s z} = 1$ if $\s z \notin [\lambda]^*$.
\end{thm}
\begin{proof}
 The first equality is just \eqref{sh-branch}. To get the second equality, note that the left-hand side enumerates all pairs $(\s z,F)$, where $\s z$ is in the first row and $F \in \p F$. So we have to prove that for every $\s c = (r,s) \in \p C^*[\lambda]$, the term on the right-hand side corresponding to $(r,s)$ enumerates all arrangements $G \in \p G_{\s c}$ and $s(G) = (1,j)$ for some $j$, $1 \leq j \leq \lambda_1$. If $r = 1$, this is true for all $\p G_{\s c}$, and there are 
 $$\prod_{\stackrel{(i,j) \in [\lambda]\setminus \p C[\lambda]}{\scriptscriptstyle i \neq 1}}(h_{ij}^* - 1) \prod_{j = 1}^{s-1} h^*_{1j} = \left( h^*_{1,r-1} + h^*_{1s} - 1 \right){ \prod_{\stackrel{(i,j) \in [\lambda]\setminus \p C[\lambda]}{\scriptscriptstyle i \neq r,j \neq r-1,s}} (h_{ij}^* - 1)\prod_{i=2}^{r-1} h^*_{is} \prod_{j=r}^{s-1} h^*_{rj} \prod_{i=2}^{r-1} h^*_{i,r-1}}$$
 such arrangements, where we are using the fact that $h^*_{10} = 1$ by definition and $h^*_{1s} = 1$.\\
 If $r > 1$, $G$ satisfies $s(G) = (1,j)$ if and only if $1 \in A(G) \cup C(G)$, by (S\ref{s0}). We can therefore put any labels in $(1,r-1)$ and $(1,s)$ except $(1,r-1)$ and $(1,s)$ simultaneously. That means that we have $h^*_{1,r-1} + h^*_{1,s} - 1$ possible labels in these squares. It follows that there are 
 $$\left( h^*_{1,r-1} + h^*_{1s} - 1 \right){ \prod_{\stackrel{(i,j) \in [\lambda]\setminus \p C[\lambda]}{\scriptscriptstyle i \neq r,j \neq r-1,s}} (h_{ij}^* - 1)\prod_{i=2}^{r-1} h^*_{is} \prod_{j=r}^{s-1} h^*_{rj} \prod_{i=2}^{r-1} h^*_{i,r-1}}$$
 such arrangements, as required.\\
 The last proof is a bit more involved. We want to see that a term on the right-hand side enumerates all $G$ so that $s(G) = (1,1)$. If $r = 1$, this is true if and only if $G_{1,1} = (1,1)$, and there are 
 $$\prod_{\stackrel{(i,j) \in [\lambda]\setminus \p C[\lambda]}{\scriptscriptstyle i \neq 1}}(h_{ij}^* - 1) \prod_{j = 2}^{s-1} h^*_{1j}$$
 such arrangements. If $r = 2$, we also have $s(G) = (1,1)$ if and only if $G_{1,1} = (1,1)$, and there are
 $$\prod_{\stackrel{(i,j) \in [\lambda]\setminus \p C[\lambda]}{\scriptscriptstyle i \neq 2,j \neq 1,s}} (h_{ij}^* - 1) \left[h^*_{1s} \prod_{j = 2}^{s-1} h^*_{2j}\right].$$
 If $r \geq 3$, we claim the following. Given labels $G_{\s z}$ for $\s z \notin \set{(1,r-1),(1,s),(2,r-1),(2,r)}$, there are $h^*_{1s}h^*_{2,r-1} + (h^*_{1,r-1} - 1)(h^*_{2s}-1)$ ways to add labels in $(1,r-1),(1,s),(2,r-1),(2,r)$ so that the resulting $G$ satisfies $s(G) = (1,1)$.\\
 Consider the rows from $3$ to $r-1$ of the shaded columns. If they are empty or form an even block or start with an even block-right dot or odd block-left dot, a careful application of rules (S\ref{s0})--(S\ref{s5}) shows that the top two rows can form any of the combinations shown on the top of Figure \ref{fig19}. 
 
 \begin{figure}[hbt]
\psfrag{G}{$G$}
\begin{center}
\epsfig{file=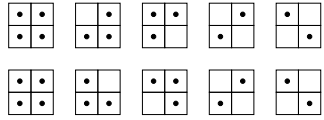,height=1.5cm}
\end{center}
\caption{Possible top two rows.}
\label{fig19}
\end{figure}

There are
 $$1 + (h_{1,r-1}^* - 1) + (h_{2s}^*-1) + (h_{1,r-1}^*-1)(h_{2s}^*-1) + (h_{1s}^*-1)(h_{2,r-1}^*-1)$$
 such labels. If, on the other hand, the rows from $3$ to $r-1$ of the shaded columns form an odd block or start with an odd block-right dot or even block-left dot, the top two rows can form any of the combinations on the bottom drawing in Figure \ref{fig19}, and there are 
 $$1 + (h_{1s}^* - 1) + (h_{2,r-1}^*-1) + (h_{1,r-1}^*-1)(h_{2s}^*-1) + (h_{1s}^*-1)(h_{2,r-1}^*-1)$$
 possible labels. Using the formula $(h_{1s}^* - 1) + (h_{2,r-1}^*-1) = (h_{1,r-1}^* - 1) + (h_{2s}^*-1)$ and simplifying, we get that both these expressions equal $h^*_{1s}h^*_{2,r-1} + (h^*_{1,r-1} - 1)(h^*_{2s}-1)$. There are therefore
 $$\left( { h^*_{1s}h^*_{2,r-1} + (h^*_{1,r-1} - 1)(h^*_{2s}-1)} \right){ \prod_{\stackrel{(i,j) \in [\lambda]\setminus \p C[\lambda]}{\scriptscriptstyle i \neq r,j \neq r-1,s}} (h_{ij}^* - 1)\prod_{i=3}^{r-1} h^*_{is} \prod_{j=\max\set{r,2}}^{s-1} h^*_{rj} \prod_{i=3}^{r-1} h^*_{i,r-1}}$$
 such arrangements, and it is easy to check that this formula is compatible with formulas for $r = 1$ and $r = 2$ obtained above.
\end{proof}

While weighted formulas were at the very center of \cite{ckp} and \cite{konva}, we mention them here only in passing. The reason is that the formulas we were able to obtain do not seem to simplify the bijection, unlike in the non-shifted case; in fact, because monomials have coefficients other than $1$, the bijection becomes more unwieldy. For a discussion of whether or not these weighted formulas are satisfactory and whether a better generalization might exist, see Section \ref{final}.

\medskip

The following figure shows how we weight the punctured hook of a square in the shifted diagram.

\begin{figure}[hbt]
\psfrag{f1g}{$23$}

\psfrag{f2g}{$18$}

\psfrag{f3g}{$18$}

\psfrag{f4g}{$57$}

\psfrag{f5g}{$15$}

\psfrag{f6g}{$19$}

\psfrag{f7g}{$18$}

\psfrag{f8g}{$19$}

\psfrag{f9g}{$39$}

\psfrag{f10g}{$\phantom{11}$}

\psfrag{f11g}{$29$}

\psfrag{f12g}{$48$}

\psfrag{f13g}{$44$}

\psfrag{f14g}{$28$}

\psfrag{f15g}{$66$}

\psfrag{f16g}{$77$}

\psfrag{f17g}{$28$}

\psfrag{f18g}{$39$}

\psfrag{f19g}{$34$}

\psfrag{f20g}{$38$}

\psfrag{f21g}{$35$}

\psfrag{f22g}{$37$}

\psfrag{f23g}{$39$}

\psfrag{f24g}{$38$}

\psfrag{f25g}{$\phantom{11}$}

\psfrag{f26g}{$57$}

\psfrag{f27g}{$45$}

\psfrag{f28g}{$48$}

\psfrag{f29g}{$67$}

\psfrag{f30g}{$48$}

\psfrag{f31g}{$55$}

\psfrag{f32g}{$77$}

\psfrag{f33g}{$67$}

\psfrag{f34g}{$68$}

\psfrag{f35g}{$68$}

\psfrag{f36g}{$67$}

\psfrag{f37g}{$\phantom{11}$}

\psfrag{f38g}{$\phantom{11}$}

\psfrag{f14}{$25$}

\psfrag{f17}{$68$}

\psfrag{f21}{$35$}

\psfrag{f24}{$38$}

\psfrag{f31}{$58$}

\psfrag{f34}{$58$}

\psfrag{x}{$x$}
\psfrag{y1}{$y_1$}
\psfrag{y2}{$y_2$}
\psfrag{y3}{$y_3$}
\begin{center}
\epsfig{file=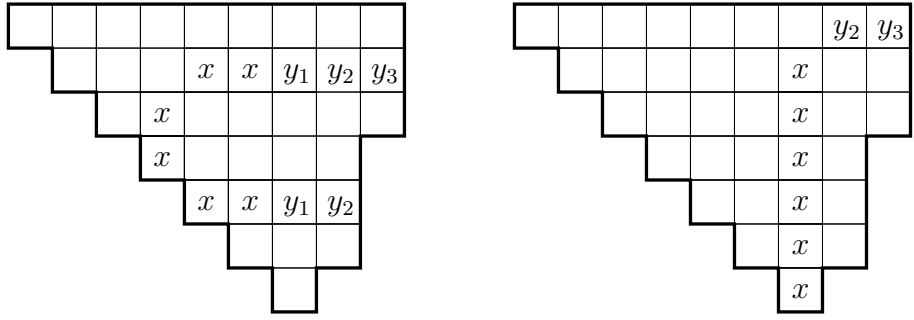,height=4.2cm}
\end{center}
\caption{Weighted punctured hooks of $(2,4)$ and $(1,7)$ in $\lambda = 9875431$ are $6x + 2y_1 + 2y_2 + y_3$ and $6x + y_1 + y_2$, respectively.}
\label{fig20}
\end{figure}

\medskip

More precisely, we define
$$\pH^*_{ij} = \Big\{  \begin{array}{r} \scriptstyle (2 \ell(\lambda) - 2 - i - j) \, x + 2 y_1 + \ldots + 2 y_{\lambda_{j+1}+j-\ell(\lambda)+1} + y_{\lambda_{j+1}+j-\ell(\lambda)+2}+ \ldots + y_{\lambda_i+ i - \ell(\lambda)}  \colon j < \ell(\lambda) \\ \scriptstyle (\max\set{k \colon \lambda_k \geq j + 1 - k \geq 1} - i)\, x + y_{j+2-\ell(\lambda)} + \ldots + y_{\lambda_i+i - \ell(\lambda)} \colon j \geq \ell(\lambda) \end{array}$$
for $(i,j) \in [\lambda]^*$.

\begin{thm} \label{weighted}
 Let $x, y_1,\ldots,y_{\lambda_1 + 1 - \ell(\lambda)}$ be some commutative variables. For a partition $\lambda$ of $n$ with distinct parts, we have the following polynomial equalities:
\begin{equation*}
 \aligned
 \left[ \sum_{(p,q) \in [\lambda]^*} x \, y_{q - \ell(\lambda) + 1} \right] \!\!\left[\prod_{(i,j) \in [\lambda]^* \setminus \p C^*[\lambda]} \!\!\!\overline{\p H}^*_{ij} \right] = \sum_{(r,s) \in \p C^*[\lambda]} \left[\prod_{\stackrel{(i,j) \in [\lambda]^* \setminus \p C^*[\lambda]}{\scriptscriptstyle i \neq r, j \neq r-1,s}} \!\!\!\overline{\p H}^*_{ij} \right] & \\
 \times \prod_{i=1}^{r} (\pH^*_{is} + x)\prod_{j=r}^{\ell(\lambda)-1} (\pH^*_{rj} + x)\prod_{j=\ell(\lambda)}^{s} (\pH^*_{rj} + y_{j-\ell(\lambda)+1}) \prod_{i=1}^{r-1} (\pH^*_{i,r-1}+x) & \\
 \left[ (\ell(\lambda) - 1) x + \!\!\!\!\sum_{j = 1}^{\lambda_1 + 1 - \ell(\lambda)} \!\!\!\! y_j \right] \!\! \left[\prod_{(i,j) \in [\lambda]^* \setminus \p C^*[\lambda]} \!\!\!\overline{\p H}^*_{ij} \right] = \!\!\sum_{(r,s) \in \p C^*[\lambda]} \left[\prod_{\stackrel{(i,j) \in [\lambda]^* \setminus \p C^*[\lambda]}{\scriptscriptstyle i \neq r, j \neq r-1,s}} \!\!\!\overline{\p H}^*_{ij} \right] \!\!\left[\Big\{ \begin{array}{r} \scriptstyle \pH^*_{1,r-1} + \pH^*_{1s} + x \colon r \geq 2\\ \scriptstyle 1 \colon r = 1\end{array}\right] & \\
 \times \prod_{i=2}^{r} (\pH^*_{is} + x)\prod_{j=r}^{\ell(\lambda)-1} (\pH^*_{rj} + x)\prod_{j=\ell(\lambda)}^{s} (\pH^*_{rj} + y_{j-\ell(\lambda)+1}) \prod_{i=2}^{r-1} (\pH^*_{i,r-1}+x) & \\
  \left[\prod_{(i,j) \in [\lambda]^* \setminus \p C^*[\lambda]} \!\!\!\overline{\p H}^*_{ij} \right] = \sum_{(r,s) \in \p C^*[\lambda]} \left[\prod_{\stackrel{(i,j) \in [\lambda]^* \setminus \p C^*[\lambda]}{\scriptscriptstyle i \neq r, j \neq r-1,s}} \!\!\!\overline{\p H}^*_{ij} \right] \!\!\left[\Bigg\{ \begin{array}{r} \scriptstyle (\pH^*_{1s} + x)(\pH^*_{2,r-1} + x) + \pH^*_{1,r-1} \pH^*_{2s} \colon r \geq 3\\ \scriptstyle \pH^*_{1s} + x \colon r = 2 \\ \scriptstyle 1 \colon r = 1\end{array}\right] & \\
 \times \prod_{i=3}^{r} (\pH^*_{is} + x)\prod_{j=\max\set{r,2}}^{\ell(\lambda)-1} (\pH^*_{rj} + x)\prod_{j=\max\set{\ell(\lambda),2}}^{s} (\pH^*_{rj} + y_{j-\ell(\lambda)+1}) \prod_{i=3}^{r-1} (\pH^*_{i,r-1}+x) & \\
  \endaligned
 \end{equation*}
\end{thm}

We omit the (rather tedious) proof; it involves noting that given an arrangement $F \in \p F$ or $G \in \p G$, we have to replace the square $\s z'$ in the hook of $\s z$ by a variable $x$ or $y_j$, using Figure \ref{fig20} or the definition of $\pH^*_{ij}$ above. Then we have to prove that the product of labels in $G$ is equal to the product of labels in $\Phi(G)$.

\medskip

The details are left as an exercise for the reader.

\section{Toward $d$-complete posets} \label{d-complete}

An interestion generalization of diagrams and shifted diagrams are $d$-complete posets of Proctor and Peterson. Let us start with a brief outline of their definition and main features. See \cite{proc} for details.

\medskip

For $k \geq 3$, take a chain of $2k-3$ elements, and expand the middle element to two incomparable elements. Call the resulting poset a \emph{double-tailed diamond poset} (for $k = 3$, we call it a \emph{diamond}), and denote it by $d_k(1)$. Figure \ref{fig22} shows the Hasse diagram of $d_6(1)$ (rotated by $90^\circ$). An interval in a poset is called a $d_k$-interval if it is isomorphic to $d_k(1)$. A $d_3^-$-interval in a poset $P$ consists of three elements $w,x,y$, so that both $x$ and $y$ cover $w$. For $k \geq 4$, an interval in $P$ is called $d_k^-$-interval if it is isomorphic to $d_k(1) \setminus \set t$, where $t$ is the maximal element of $d_k(1)$. 

\medskip

\begin{figure}[hbt]
\begin{center}
\psfrag{1}{$1$}
\psfrag{2}{$2$}
\psfrag{3}{$3$}
\psfrag{4}{$4$}
\psfrag{5}{$5$}
\psfrag{6}{$6$}
\psfrag{7}{$7$}
\psfrag{8}{$8$}
\psfrag{9}{$9$}
\psfrag{10}{$10$}
\epsfig{file=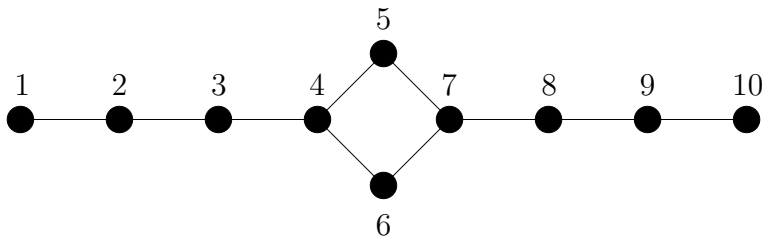,height=3cm}
\end{center}
\caption{The poset $d_5(1)$.}
\label{fig22}
\end{figure}

The essential property of $d$-complete posets is that there are we can ``complete'' every $d_k^-$-interval. Furthermore, $d_k$-intervals do not intersect. More precisely, we call a poset $(P, \leq)$ \emph{$d$-complete} if it has the following properties:
\begin{enumerate}
 \renewcommand{\labelenumi}{(D\arabic{enumi})}
 \item if $x$ and $y$ cover a third element $w$, there must exist a fourth element $z$ which covers each of $x$ and $y$;
 \item if $\set{x,y,z,w}$, $w < x,y < z$, is a diamond in $P$, then $z$ covers only $x$ and $y$ in $P$;
 \item no two elements $x$ and $y$ can cover each of two other elements $w$ and $w'$;
 \item if $[w,y]$ is a $d_k^-$-interval, $k \geq 4$, there exists $z$ which covers $y$ and such that $[w,z]$ is a $d_k$-interval;
 \item if $[w, z]$ is a $d_k$-interval, $k \geq 4$, then $z$ covers only one element in $P$;
 \item if $[w,y]$ is a $d_k^-$-interval, $k \geq 4$, and $x$ is the (unique) element covering $w$ in this interval, then $x$ is the unique element covering $w$ in $P$.
\end{enumerate}

The \emph{hook length} of an element $z$ of a $d$-complete poset $P$ is defined recursively as follows. If $z$ is not an element of a $d_k$-interval for any $k \geq 3$, then the hook length $h_z^P$ is $|\set{w \in P \colon w \leq z}|$. If there exists $w$ so that $[w,z]$ is a $d_k$-interval, and $x,y$ are the incomparable elements in $[w,z]$, then define $h_z^P = h_x^P + h_y^P - h_w^P$. We usually write $h_z$ for $h_z^P$ if the choice of the poset is clear.

\medskip

We can also define the \emph{hook} $H_z$ of an element $z$ in a $d$-complete poset $P$ so that $H_z \subseteq \set{w \in P \colon w \leq z}$, $z \in H_z$ and $|H_z| = h_z$. One possible definition goes as follows. Assume without loss of generality that $P$ is connected. Then it is easy to see that $P$ has a maximal element $z_0$. Pick $z$ and $w \leq z$. It can be proved that $h_z^{(w,z_0]}$ either equals $h_z^{[w,z_0]}$ or $h_z^{[w,z_0]}-1$ (note that the posets $(w,z_0]$ and $[w,z_0]$ are also $d$-complete). We say that $w \in H_z$ if and only if $h_z^{(w,z_0]}=h_z^{[w,z_0]}-1$. As usual, we write $\overline H_z$ for $H_z \setminus \set z$.

\medskip

One of the main properties of $d$-complete posets is that the hook length formula is still valid. If $f_P$ is the number of bijections $g \colon P \to [n]$ satisfying $g(z) \leq g(w)$ for $z \geq w$, then 
$$f_P = \frac{n!}{\prod_{z \in P} h_z}.$$

\medskip

Denote the set of all minimal elements of $P$ by $\min P$. By induction, the hook length formula is equivalent to the \emph{branching rule for $d$-complete posets}:
\begin{equation}\label{d-branch}
 n  \cdot \prod_{z \in P \setminus \min P} (h_z - 1) \ = \sum_{c \in \min P} \ \prod_{\substack{z \in P \setminus \min P\\c \notin H_z}} (h_z - 1) \prod_{\substack{z \in P \setminus \min P \\ c \in H_z}} h_z.
\end{equation}

In analogy with our framework for non-shifted and shifted diagrams, define $\p F$ as the set of all pairs $(z,F)$, where $z \in P$ and $F \colon P \setminus \min P \to P$ satisfies $F(z) \in \overline H_z$ for all $z$. Furthermore, define $\p G$ as the set of all pairs $(c,G)$, where $c \in \min P$ and $G \colon P \setminus \min P \to P$ satisfies $G(z) \in \overline H_z$ for all $z$, $c \notin H_z$, $G(z) \in H_z$ for all $z$, $c \in H_z$. For $c \in \min P$ and $k \geq 0$, denote by $\p G_c$ the set of all $G$ so that $(c,G) \in \p G$ and by $\p G_c^k$ the set of all $G \in \p G_c$ with $|\set{z \in P \colon G(z) = z}| = k$. Sometimes we write $F_z$ and $G_z$ instead of $F(z)$ and $G(z)$. For a $d$-complete poset $P$ and $c \in \min P$, denote by $A_c = A_c(G)$ the set $\set{z \in P \colon c \in \overline H_z}$. A proof of the following conjecture would provide a bijective proof of the hook length formula for $d$-complete posets.

\begin{conj}
 For a $d$-complete poset $P$ and a minimal element $c$, there exist maps $s = s_c \colon \p P(A_c) \to P$, which can be extended to $s \colon \p G_c \to P$ by $s(G) = s(A_c(G))$, and $e = e_c\colon \p G_{c} \to \p G_{c}$ satisfying the following properties:
 \begin{enumerate}
  \renewcommand{\labelenumi}{(P\arabic{enumi})}
  \item if $G \in \p G_{c}^0$, then $s(G) =c$ and $e(G) = G$, and if $G \in \p G_{c}^k$ for $k \geq 1$, then $e(G) \in \p G_{c}^{k-1}$;
  \item if $G \in \p G_{c}^k$ for $k \geq 1$, then $e(G)_{s(G)} = s(e(G))$; furthermore, if $z \not\geq s(G)$, then $e(G)_{z} = G_{z}$;
  \item given $G' \in \p G_{c}^k$ and $z$ for which $G'_{z} = s(G')$, there is exactly one $G \in \p G_{c}^{k+1}$ satisfying $s(G) = z$ and $e(G) = G'$.
 \end{enumerate}
\end{conj}

The following proposition presents some evidence for the conjecture. For the definition of a \emph{slant sum}, see \cite[page 67]{proc}.

\begin{prop}
 Let $P,P_1,P_2$ be $d$-complete posets.
 \begin{enumerate}
  \renewcommand{\labelenumi}{(\alph{enumi})}
  \item The conjecture holds for $P$ if $P$ is the diagram of a partition.
  \item The conjecture holds for $P$ if $P$ is the shifted diagram of a partition.
  \item The conjecture holds for $P$ if $P$ is a double-tailed diamond poset.
  \item If the conjecture holds for $P_1$ and $P_2$, it also holds for their disjoint union.
  \item If the conjecture holds for $P_1$ and $P_2$, it also holds for their slant sum.
 \end{enumerate}
\end{prop}
\begin{skt}
 We showed (a) in Section \ref{non-shifted} and (b) in Sections \ref{basic} and \ref{details}.\\
 Take $P=d_k(1)$ for $k \geq 3$ (for $k = 3$, this is the diagram of the partition $22$, and for $k = 4$, this is the shifted diagram of the partition $321$). Label the elements by $x_1,\ldots,x_{k-1},y_1,\ldots,y_{k-1}$ in such a way that $x_1 > x_2 > \ldots >  x_{k-1},y_{k-1} > y_{k-2} > \ldots > y_1$. The only minimal element is $c = y_1$, and $A_c = P \setminus \set{x_1,y_1}$. It is also important to note that for all $j = 2,\ldots,k-1$, we have an identification of $\overline H_{x_1}$ and $\overline H_{x_j} \sqcup \overline H_{y_j}$, for example
 $$\setcounter{MaxMatrixCols}{13}
 \begin{matrix}
 \overline H_{x_1}: & x_2 & \cdots & x_j & x_{j+1} & \cdots & x_d & y_d & \cdots & y_{j+1} & y_j & \cdots & y_2 \\
 \overline H_{x_j}: &     &        &     & x_{j+1} & \cdots & x_d & y_d & \cdots & y_{j+1} & y_{j-1} & \cdots & y_1 \\
 \overline H_{y_j}: & y_1 & \cdots & y_{j-1} & & & & & & & &
 \end{matrix}$$
 For $A \subseteq A_c$, define
 $$s(A) = \left\{ \begin{array}{ccl} y_1 & : & A = \emptyset \\ \max A & : & A \neq \emptyset, \, \min\set{i \colon x_i \in A} \neq \max\set{i \colon y_i \in A} \\ x_1 & : & \min\set{i \colon x_i \in A} = \max\set{i \colon y_i \in A} \end{array} \right..$$
 Furthermore, if $A = \emptyset$, define $e(G) = G$; if $A \neq \emptyset$ and $\min\set{i \colon x_i \in A} \neq \max\set{i \colon y_i \in A}$, define $A' = A \setminus \set{\max A}$, $e(G)_{\max A} = s(A')$ and $e(G)_z = G_z$ for $z \neq \max A$. Finally, if $\min\set{i \colon x_i \in A} = \max\set{i \colon y_i \in A} = j$, we have $G_{x_1} \in \overline H_{x_j} \sqcup \overline H_{y_j}$. If $G_{x_1} \in \overline H_{x_j}$, define $A' = A \setminus \set{x_j}$, $e(G)_{x_1} = s(A')$, $e(G)_{x_j} = G_{x_1}$, $e(G)_z = G_z$ for $z \neq x_1,x_j$; and if $G_{x_1} \in \overline H_{y_j}$, define $A' = A \setminus \set{y_j}$, $e(G)_{x_1} = s(A')$, $e(G)_{y_j} = G_{x_1}$, $e(G)_z = G_z$ for $z \neq x_1,y_j$. It can be checked that such $s$ and $e$ satisfy (P\ref{p1}), (P\ref{p2}), and (P\ref{p3}). This proves (c).\\
 Suppose $P$ is a disjoint union of $d$-complete posets $P_1$ and $P_2$ with corresponding maps $s_1,e_1$, $s_2,e_2$. If $c \in P_i$, then $\set{z \in P \colon G(z) = z} \subseteq P_i$, and therefore we should take
 $$s(G) = s_i(G|_{P_i}),\qquad e(G)_z = \left\{\begin{array}{ccl} e_i(G|_{P_i})_z & : & z \in P_i \\ G_z & : & z \notin P_i \end{array} \right..$$
 Again, it can be checked that such $s$ and $e$ satisfy (P\ref{p1}), (P\ref{p2}), and (P\ref{p3}), and this proves (d).\\
 Suppose $P$ is a slant sum of $d$-complete posets $P_1$ and $P_2$ with corresponding maps $s_1,e_1$, $s_2,e_2$. If $c \in P_1$, or if $c \in P_2$ and $\set{z \in P \colon G(z) = z} \subseteq P_2$, the same construction as in (d) works. If, however, $c \in P_2$ and $\set{z \in P \colon G(z) = z} \not\subseteq P_2$, we have to take
 $$s(G) = \max \set{z \in P \colon G(z) = z},\qquad e(G)_z = \left\{ \begin{array}{ccl} G_z & : & z \neq s(G) \\ s(A') & : &  z = s(G) \end{array} \right.,$$
 where $A' = A \setminus \set{s(G)}$. It can be checked that such $s$ and $e$ satisfy (P\ref{p1}), (P\ref{p2}), and (P\ref{p3}). This proves (e) and finishes the proof of the proposition. \qed
\end{skt}

Proctor \cite{proc} gave a classification of irreducible $d$-complete posets. So the proof of the conjecture would be complete once we found maps $s_c$ and $e_c$ for all minimal elements of all other classes of irreducible $d$-complete posets. While some of the cases are pretty straightforward, the map $s$ is likely to be extremely complicated for, say, the ``bat'' and its unique minimam element.

\section{Final remarks} \label{final}

\subsection{} Our weighted formulas are slightly disappointing, since all the ``height'' variables $x_1,\ldots,x_{\ell(\lambda)}$ have merged into one variable, which also appears in columns $1,\ldots,\ell(\lambda)-1$. There are certain indications, however, that this is the best possible generalization (the author would be thrilled to be proved wrong though); let us mention two. First, the weighted punctured hooks $\pH_{ij}$ should satisfy the formula $\pH^*_{ij} = \pH^*_{i,r-1} + \pH^*_{j+1,s} = \pH^*_{i,s} + \pH^*_{j+1,r-1}$, and it should be obvious from Figure \ref{fig13} that this is not easily obtainable unless many of the variables are equal. Also, for the staircase shape partition $(k,k-1,\ldots,1)$, the ``natural'' weighted generalization of $n$ on the left-hand side of \eqref{sh-branch} is $\sum_{i \leq j} x_i y_j$, which should be factored into linear terms on the right-hand side, and this is obviously impossible. Our weighted version, in which both sides are just integer multiples of a power of $x$, obviously avoids this issue.

\subsection{} There is another variant of \eqref{sh-branch} obtained by finding all formations of the shaded columns that satisfy $s(A,B,C) = (1,m)$ for $1 \leq m \leq \lambda_1$: 

$$\prod_{\scriptscriptstyle(i,j) \in [\lambda]^* \setminus \p C^*[\lambda]}  \left({ h^*_{ij} - 1}\right) = \sum_{\scriptscriptstyle(r,s) \in \p C^*[\lambda]}H_{\lambda,r,s,m}\!\!\!{\prod_{\stackrel{(i,j) \in [\lambda]^*\setminus \p C^*[\lambda]} {\scriptscriptstyle i \neq r,j \neq r-1,s}} \!\!\!\!\!(h_{ij}^* - 1)\prod_{i=m+2}^{r-1} h^*_{is} \!\!\!\! \!\!\prod_{j=\max\set{r,m+1}}^{s-1} \!\!\!\!\!\! h^*_{rj} \prod_{i=m+2}^{r-1} h^*_{i,r-1}},$$
where $H_{\lambda,r,s,m}$ is
$$\aligned
& \prod_{i = 1}^{r-1} (h_{i,r-1}^*\! - \!1) \prod_{j = r}^{m-1} (h_{r,j}^*\! - \!1) \prod_{i = 2}^{r-1} (h_{i,s}^*\! - \!1) \qquad\qquad\qquad\qquad\qquad\qquad\qquad\qquad\mbox{ if } 1 \leq r \leq m \\  
& \sum_{\stackrel{I \subseteq [m],k \in I}{\scriptscriptstyle \min I = 1}} \, \,  \prod_{i \notin I \mbox{\tiny{ or }} i > k} \!\!\!(h^*_{i,r-1}\! - \!1) \!\!\!\!\prod_{i \notin I \mbox{\tiny{ or }} i < k}\!\!\! (h^*_{i,s} \! - \! 1) + \prod_{i = 2}^{r-1} h_{i,r-1}^*  \prod_{i = 1}^{r-1} (h_{i,s}^*\! - \!1) \qquad\qquad\quad\:\!\mbox{ if } r = m + 1 \\ 
& \!\!\! \sum_{\stackrel{I \subseteq [m],k \in I}{\scriptscriptstyle \min I = 1}} \, \, \prod_{i \notin I \mbox{\tiny{ or }} i > k} \!\!\!(h^*_{i,r-1}\! - \!1) \!\!\!\!\prod_{i \notin I \mbox{\tiny{ or }} i < k}\!\!\! (h^*_{i,s} \! - \! 1) + \!\!\!\!\!\!\!\!\!\!\! \sum_{\stackrel{I \subseteq [m+1],k \in I}{\scriptscriptstyle \min I = 1,\max I = m+1}}  \prod_{i \notin I \mbox{\tiny{ or }} i > k} \!\!\!(h^*_{i,r-1}\! - \!1) \!\!\!\!\prod_{i \notin I \mbox{\tiny{ or }} i < k}\!\!\! (h^*_{i,s} \! - \! 1) \:\:+ \\ 
& (h^*_{m+1,r-1}\! - \!1)\prod_{i = 2}^{m} h_{i,r-1}^*  \prod_{i = 1}^{m} (h_{i,s}^*\! - \!1) + (h^*_{m+1,s}\! - \!1) \prod_{i = 2}^{m} h_{i,s}^*  \prod_{i = 1}^{m} (h_{i,r-1}^*\! - \!1)  \qquad\:\,\mbox{ if } r \geq m + 2
\endaligned$$

For $m = 1$, this is, of course, the third equality of Theorem \ref{vars}, and the proof for a general $m$ is very similar. Just let us note that
$$\sum_{\stackrel{I \subseteq [m],k \in I}{\scriptscriptstyle \min I = 1}} \, \, \prod_{i \neq I \mbox{\tiny{ or }} i > k} \!\!\!(h^*_{i,r-1}\! - \!1) \!\!\!\!\prod_{i \neq I \mbox{\tiny{ or }} i < k}$$
counts the number of ways in which we can choose labels in $[1,m] \times \set{r-1,s}$ so that the shaded columns up to row $r$ form a left snake that starts in row $1$.

\medskip

A weighted version is also possible, but is omitted. While this formula seems to be too complicated to be of real interest, we should mention that it indeed gives $\lambda_1-1$ formulas that are not equivalent to the formulas in Theorem \ref{vars}. In the non-shifted case, finding all $G$ satisfying $s(G) = (1,m)$ would yield formulas that are equivalent to the case $m = 1$ for the partition $(\lambda_1-m,\lambda_2-m,\ldots)$.

\medskip

We can similarly enumerate all $G$ so that $s(G) = (i,m)$ for some fixed $m$. The resulting formula, which we omit, has $m \cdot \prod_{\scriptscriptstyle(i,j) \in [\lambda]^* \setminus \p C^*[\lambda]}  ({ h^*_{ij} - 1})$ on the left-hand side, and on the right, the sticks and the snakes need not start in row $1$.

\subsection{} We were not able to find ``complementary'' formulas in the spirit of \cite{konva}, and we leave them as an open problem. The idea, following \cite[\S 6]{konva}, should be to look at the ``complement'' of the given shifted tableau (see Figure \ref{fig21}), whose corners are the ``outer'' corners of the original partition; to write out the formulas from Theorem \ref{vars} or Theorem \ref{weighted} for this complementary partition; to cancel out the terms that the left-hand side and the right-hand side have in common; and then, for each square $\s z^*$ appearing in such reduced formula, to find a square $\s z$ of the original shifted diagram with satisfying $h_{\s z} = h_{\s z^*}$ or $h_{\s z} +1 = h_{\s z^*} - 1$; and to finally change this ``reduced'' complementary formula into a formula that involves products over all squares. The author was not able to follow through with this approach.

 \begin{figure}[hbt]
\psfrag{G}{$G$}
\begin{center}
\epsfig{file=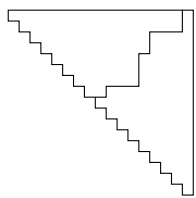,height=2.5cm}
\end{center}
\caption{Complementary shifted partition.}
\label{fig21}
\end{figure}

\subsection{} Hook walk proofs of Theorem \ref{weighted} should be possible as a generalization of Sagan's proof \cite{sagan}, or probably in a more intuitive and less technical way, using maps $s$ and $e$ as guiding lights. We leave this as an excercise for the (very determined) reader. Such proofs might also work for the (hypothetical) complementary formulas.

\subsection{} It was mentioned in \cite[\S 4]{konva} that variants of the branching rule give new recursions for $f_\lambda$, see \cite[Corollary 5]{konva}. This is, of course, also true for the shifted case. For example, the second equality of Theorem \ref{vars} has the following corollary:
$$\lambda_1 f^*_\lambda = \sum_{\s c = (r,s) \in \p C^*[\lambda]} n \left( {\textstyle \frac 1{\lambda_1 + r - s} + \frac 1{\lambda-r + s + 1} - \frac 1{(\lambda_1 + r - s)(\lambda-r + s + 1)}}\right) f^*_{\lambda-\s c}.$$
Again, it would be interesting to know if this identity has a combinatorial meaning.

\bigskip

{\bf Acknowledgements}

\end{document}